\documentclass[12pt,a4paper,reqno]{amsart}
\usepackage[vmargin=2cm, hmargin=2cm]{geometry}
\usepackage{amsfonts,amsthm,amsmath,amssymb,amscd}
\usepackage{bbm}

\usepackage{enumerate}

\usepackage{xcolor}

\theoremstyle{definition}
\newtheorem{dfn}{Definition}

\newtheorem{rem}[dfn]{Remark}

\theoremstyle{plain}
\newtheorem{lem}[dfn]{Lemma}
\newtheorem{thm}[dfn]{Theorem}

\newtheorem{prop}[dfn]{Proposition}

\newtheorem{cor}[dfn]{Corollary}

\newcommand{\C}{{\mathbb C}}
\newcommand{\R}{{\mathbb R}}
\newcommand{\N}{{\mathbb N}}
\newcommand{\1}{{\mathbbm 1}}
\DeclareMathOperator{\diam}{diam}
\DeclareMathOperator{\ord}{ord}
\DeclareMathOperator{\Int}{Int}
\DeclareMathOperator{\id}{id}
\newcommand{\oc}{\widehat{\C}}

\numberwithin{equation}{section}

\def\cS{\mathcal S}
\def\om{\omega}
\def\sg{\sigma}
\def\({\big(} \def\){\big)}
 \def\bu{\bigcup}
\def\es{\emptyset}
\def\du{\bigoplus}
\def\lra{\longrightarrow}
\def\sbt{\subset}  
\def\Ga{\Gamma} \def\a{\alpha} \def\sg{\sigma}
\def\ov{\overline}
\def\sms{\setminus}
\def\P{{\rm P}}
\def\Crit{{\rm Crit}}
\def\PC{{\rm PC}}
\def\PS{{\rm PS}}
\def\h{{\rm h}}
\def\H{{\rm H}}
\def\cC{\mathcal C}
\def\Sing{{\rm Sing}}
\def\dist{{\rm dist}}
\def\lam{\lambda}
\def\al{\alpha}
\def\cL{{\mathcal L}}
\def\fr{\noindent}

\title[Long hitting times]{Long hitting times for expanding systems}

\thanks{We would like to thank Stefano Galatolo and Juan Rivera--Letelier, whose bibliographical remarks allowed us respectively to improve our citations and to weaken the hypotheses of some of our application results. We are also indebted to Roland Zweim\"{u}ller for discussions on the subject and to the reviewers for their helpful comments.}

\begin{document}

\author{{\L}ukasz Pawelec}
\address{{\L}ukasz Pawelec, Department of Mathematics and Mathematical Economics, Warsaw School of Economics, 
al.~Niepodleg\l{}o\'{s}ci~162, 02-554 Warszawa, Poland}
\email{LPawel@sgh.waw.pl}

\author{Mariusz Urba\'{n}ski}
\address{Mariusz Urba\'{n}ski, University of North Texas, Department of Mathematics, 1155 Union Circle \#311430, Denton, TX 76203-5017,
USA}
\email{urbanski@unt.edu}

\date{}
\begin{abstract}
We prove a new result in the area of hitting time statistics. Currently, there is a lot of papers showing that the first entry times into cylinders or balls are often faster than the Birkhoff's Ergodic Theorem would suggest. We provide an opposite counterpart to these results by proving that the hitting times into shrinking balls are also often much larger than these theorems would suggest, by showing that for many dynamical systems
$$
\displaystyle \limsup_{r\to 0} \tau_{B(y,r)}(x)\mu(B(y,r))=+\infty,
$$
for an appropriately large, at least of full measure, set of points $y$ and $x$.

We first do this for all transitive open distance expanding maps and Gibbs/equilibrium states of H\"older continuous potentials; in particular for all irreducible subshifts of finite type with a finite alphabet. Then we prove such result for all finitely irreducible subshifts of finite type with a countable alphabet and Gibbs/equilibrium states for H\"older continuous summable potentials. Next, we show that the \emph{limsup} result holds for all graph directed Markov systems (far going natural generalizations of iterated function systems) and projections of aforementioned Gibbs states on their limit sets. By utilizing the first return map techniques, we then prove the \emph{limsup} result for all tame topological Collect--Eckmann multimodal maps of an interval, all tame topological Collect--Eckmann rational functions of the Riemann sphere, and all dynamically semi--regular transcendental meromorphic functions from $\C$ to $\oc$.
\end{abstract}
\maketitle

\section{Introduction}

In the recent years there has been a growing interest in the topic of quantitative recurrence and hitting (also called \emph{entry}) times. There is currently a number of ways in which we may estimate the \emph{speed} of entry/return times into various sets. Up to our best knowledge, the historically first approach comes from the paper by M. Boshernitzan \cite{Bosh}. It asserts that if $(X,d)$ is a separable metric space and $(T,\mu)$ is any transformation preserving a Borel, probability measure, then 
\begin{equation}\label{wstep}
\liminf_{n\to \infty} \;n^{1/\beta}d(T^n(x),x) < +\infty, \mbox{\;\;$\mu$--a.e.}
\end{equation} 
if $\H_\beta(X)$, the $\beta$--dimensional Hausdorff measure of $X$, is $\sigma$-finite. 

Another way of quantifying the recurrence (resp. hitting) speed utilizes the notion of lower and upper recurrence (resp. hitting) rates. The hitting rates of a point $y$ into neighbourhoods of $x$ are defined as follows
\[
\underline{E}(x,y)
:=\liminf_{r\to 0}\frac{\log\tau_{B(x,r)}(x)}{-\log(r)} \mbox{\hskip0.5cm and \hskip0.5cm} 
\overline{E}(x,y)
:=\liminf_{r\to 0}\frac{\log\tau_{B(x,r)}(x)}{-\log(r)},
\]
where 
$$
\tau_U(x):=\inf\{k\geq 1: T^k(x)\in U\},
$$
is the first entry time of point $x$ into $U$. The recurrence rates are defined on the diagonal, i.e $\underline{R}(x)=\underline{E}(x,x)$ and $\overline{R}(x)=\overline{E}(x,x)$.

For many systems exhibiting some kind of hyperbolic behaviour we have 
\[
\underline{R}(x)=\underline{d}_\mu(x)
:=\liminf_{r\to 0}\frac{\log\mu(B(x,r))}{\log(r)}  \mbox{\quad and \quad} \overline{R}(x)=\overline{d}_\mu(x)
:=\limsup_{r\to 0}\frac{\log\mu(B(x,r))}{\log(r)},
\]
for $\mu$--a.e. $x$.
The respective quantities ${\underline{d}}_\mu(x)$ and ${\overline{d}}_\mu(x)$ are commonly called the lower and upper pointwise dimension of the invariant measure $\mu$ at the point $x$, for more details see \cite{BS}.

There seem to be relatively less results concerning the hitting times rates, but the two notions are usually related (at least for ergodic systems). For some general relations between the two, see \cite{HLV}

Calculating the recurrence/hitting time rates is thus related to calculating the Hausdorff (or packing) dimension of the ambient space. A more subtle task, which is a reminiscent of finding the value of the (appropriate dimensional) Hausdorff measure of the space, is to study the expression
\[
R_r(x):=\tau_{B(x,r)}(x)\cdot \mu(B(x,r))
\]  
for the return time, and 
\[
E_r(x):=\tau_{B(y,r)}(x)\cdot \mu(B(y,r))
\] 
for the entry time.

If we take a sequence $(r_n)_{n=1}^\infty$ of radii, typically converging to $0$, then $\(R_{r_n}\)_{n=1}^\infty$ may be viewed as a sequence of random variables on the probability space $(X,\mu)$. The entry times $E_{r_n}$, $n\ge 1$, may be also treated as real--valued random variables whose distribution is determined by conditional measures $\mu_{B(y,r_n)}$ on $B(y,r_n)$. 

\medskip

Returning to the beginning: it is easy to see that the aforementioned Boshernitzan's result gives the following implication:
\[ 
\mu(B(x,r))\leq Cr^\beta \  \  \mu-a.e. \  \  \implies \  \ \liminf_{r\to 0} R_r(x) <+\infty \ \  \mu-a.e.
\]
There are several results improving or widening this implication. See for example \cite{Ju} for results on generic maps on a manifold, or \cite{BGI}, where the the limit in \eqref{wstep} is proved to be always finite for some $\beta\in(0,+\infty)$. Also, the first author of the current paper proved both a strengthening of the recurrence  result \eqref{wstep} and an analogous result for the hitting times in \cite{Pa}.

Another way of looking at the hitting times statistics is to consider the  \emph{shrinking target} setup. Indeed, given $(r_n)_{n=1}^\infty$, decreasing sequence of positive real numbers, define 
\[
E\(x,(r_n)_{n=1}^\infty\):=\big\{y\in X : d(T^nx,y)<r_n \mbox{\;\; for infinitely many $n$}\big\}.
\]
The most natural question is then about the value of $\mu\(E\(x,(r_n)_{n=1}^\infty\)\)$. It is fairly obvious that if the sequence $r_n$ converges to zero sufficiently fast, then usually $\mu(E\(x,(r_n)_{n=1}^\infty\)\)=0$. Then the next natural question is about 
the Hausdorff dimension of this set. For example, the authors in \cite{Rams} consider the sets $E\(x,(r_n)_{n=1}^\infty\)$ for some fairly general maps of an interval and provide a closed formula for their Hausdorff dimension depending on the sequence $(r_n)_{n=1}^\infty$.

\medskip

It should be emphasized at this point that most of the results cited above pertain to the lower limit \eqref{wstep}, i.e. show that the recurrence is (may be) significantly faster than the one suggested by the average. Boshernitzan's formulation of the rate of recurrence is however unsuitable for the question of how slow the recurrence can be since replacing the lower limit by  the upper limit in \eqref{wstep} trivially gives $+\infty$ for all non--fixed points $x$ of the map $T$. On the other hand, studying the upper limits of $R_r(x)$ and $E_r(x)$ as $r\to 0$ makes perfect sense and constitutes a very natural problem. 

There is also a pure probabilistic counterpart to such questions. For example, taking a sequence of independent coin flips, one may ask about the longest time it will take for $n$ consecutive \emph{heads} to appear. Asking such questions started in the 1970's by, amongst others, P. Erd\"os in \cite{Erd}. This is still a popular topic in probability theory.

In the current paper we identify a large class of naturally defined dynamical systems for which 
\begin{equation}\label{eq:intres}\limsup_{r\to0}E_r(x)=+\infty,\end{equation} for almost every $x$ and $y$. Since for these systems the lower limit is equal to 0, the deviation from the expected value is as large as possible in both directions.

There are only a few results in the literature pertaining to long hitting times. As mentioned above, the recurrence rates (defined above) are often equal to the pointwise dimension, giving a \emph{logarythmic} upper and lower bound on the recurrence rate. Notably, S.~Galatolo and P. Peterlongo \cite{GP} constructed a very special system on the three dimensional torus for which both the \emph{liminf} and the \emph{limsup} are infinite $\mu$--a.e. In fact, the correct scaling in their example is not the inverse of the measure. However, as it is shown in their paper, such phenomena may appear only for slowly mixing systems. Similar results have been proved for systems with a random component, see \cite{GRS}.

\medskip

Our paper is organized as follows. Dealing all the time with limsup, in the next section, Section~\ref{sec:def}, we prove our main theorem, i.e. that~\eqref{eq:intres} holds, in the case of all open transitive distance expanding maps and Gibbs/equilibrium states of all H\"older continuous potentials; in particular for all irreducible subshifts of finite type with a finite alphabet. But we go beyond this case. In Section~\ref{sec:subshifts} we show this result also holds for symbolic systems. In Section~\ref{GDMS} we prove~\eqref{eq:intres} for all graph directed Markov systems; in particular for all finitely irreducible subshifts of finite type with a countable alphabet. Then, in Section~\ref{FRMT}, we show that the upper limit of hitting times does not change when passing to the first return map. This leads, see Section~\ref{Examples}, to a multitude of examples such as tame topological Collect--Eckmann multimodal maps of an interval, tame topological Collect--Eckmann rational functions of the Riemann sphere $\oc$, and dynamically semi--regular transcendental meromorphic functions from $\C$ to $\oc$.




\section{Distance Expanding Maps}\label{sec:def}
Let $(X,\rho)$ be a compact metric space and let 
$
T\colon X\to X
$
be an open topologically transitive Lipschitz continuous distance expanding map in the sense of \cite{PUbook}. Recall that being distance expanding means that there exist $\delta>0$ and $\lambda>1$ such that
$$
\rho(T(x),T(y))\ge \lambda \rho(x,y)
$$
for all points $x$ and $y$ in $X$ such that $\rho(x,y)\le \delta$. Taking $\delta>0$ sufficiently small we will also have that for every integer $n\ge 0$ and every $x\in X$, there exists a unique continuous map 
$$
T_x^{-n}\colon B(T^n(x),4\delta)\to X
$$
such that 
$$
T^n\circ T_x^{-n}=\id_{B(T^n(x),4\delta)}
$$
and 
$$
T_x^{-n}(T^n(x))=x.
$$
In addition, by taking $\delta>0$ small enough, we will have that
\begin{equation}\label{120180614}
\rho\big(T_x^{-n}(z),T_x^{-n}(w)\big)
\le \lambda^{-n}\rho(z,w)
\le 8\delta\lambda^{-n}
\end{equation}
for all $z, w\in B(T^n(x),4\delta)$.
The map $T_x^{-n}$ will be referred in the sequel as the unique continuous inverse branch of $T^n$ defined on $B(T^n(x),4\delta)$ and sending $T^n(x)$ to $x$.

\medskip Let $\mathcal R$ be a Markov partition for $T$ (see \cite{PUbook}) with
\begin{equation}\label{220180614}
\diam(\mathcal R)<\delta.
\end{equation}
For every integer $n\ge 1$ let 
$$
\mathcal R^n:=\mathcal R\vee T^{-1}(\mathcal R) \vee\ldots \vee T^{-(n-1)}(\mathcal R).
$$
The elements of the cover $\mathcal R^n$ will be called in the sequel the cells of order $n$ generated by the partition $\mathcal R$. Also, any union $U$ of elements of $\mathcal R^n$ which cannot be represented as a union of elements of $\mathcal R^{n-1}$ will be referred to as a set of order $n$ generated by the partition $\mathcal R$. We will then write that 
\[
n=\ord(U).
\]
If we do not want/need to specify the order of $U$, we will just say that the set $U$ is generated by the Markov partition $\mathcal R$. Because of \eqref{120180614} and \eqref{220180614}, we have that
\begin{equation}\label{320180614}
\diam({\mathcal R}^n)<\delta\lambda^{-(n-1)}.
\end{equation}
It is known (see \cite{PUbook}) that for every H\"older continuous function $f\colon X\to\R$, following tradition, called a potential in the sequel, there exists a unique equilibrium measure (state) $\mu_f$ on $X$. Being an equilibrium state means that
$$
\h_{\mu_f}(T)+\int_X f\,d\mu_f=\P(f)
:=\sup\left\{\h_{\mu}(T)+\int_X f\,d\mu\right\},
$$
where the supremum is taken over all Borel probability $T$--invariant (ergodic) measures $\mu$ on $X$ and $\h_{\mu}(T)$ is the Kolmogorov--Sinai metric entropy of $T$ with respect to $\mu$. The measure $\mu_f$ is also called, for reasons explained e.g. in \cite{PUbook}, a Gibbs state for $T$ and $f$. We should also note that the quantity $\P(f)$, called the topological pressure of $f$ with respect to $T$, has 
a purely topological characterization with no measures involved. It is also known from \cite{PUbook} that there exist two constants $\alpha>0$ and $C\ge 1$, depending on $T$ and $f$ such that
\begin{equation}\label{420180614}
\mu_f(B(z,r))\le Cr^\alpha
\end{equation}
for all $z\in X$ and all radii $r>0$. We now shall prove the following technical but very useful result.

\begin{lem}\label{l1mp3}
Assume that $X$ is a compact subset of some (finitely dimensional) Euclidean space and that $T\colon X\to X$ is an open topologically transitive distance expanding map. Assume also that $\mu$ is a Gibbs/equilibrium state for a H\"older continuous potential. If $y$ is an arbitrary point of $X$, then there exists a Lebesgue measurable set $\Delta\subset (0,1)$ with the following properties
\begin{enumerate}[(a)]
	
\item $\hfill \displaystyle \lim_{\Delta\ni r\to 0}\frac{Leb(\Delta\cap(0,r))}{r}=1. \hfill$\vskip0.5em

\item 
For every $r\in\Delta$ there exists a set $R_r$, generated by the Markov partition $\mathcal R$, satisfying \vskip0.5em
\begin{enumerate}[(b1)]
\item $\hfill \displaystyle B(y,r) \subset R_r, \hfill$ \vskip0.5em
\item $\hfill \displaystyle \frac{\mu(R_r)}{\mu(B(y,r))}\le 2,\hfill$ \vskip0.5em
\item $\hfill \displaystyle \lim_{r\to 0} \ord(R_r)\mu(B(y,r))=0. \hfill$
\end{enumerate}
\end{enumerate}
\end{lem}

\begin{proof}
Because of Lemma~3.6 in \cite{PUZI}, applied with $\kappa_y$ identically equal to $2$, there exists a Lebesgue measurable set 
$\Delta\subset (0,1)$ such that item (a) above holds and 
\begin{equation}\label{520180614aa}
\frac{\mu(B(y,r+r^2))}{\mu(B(y,r))}\le 2
\end{equation}
for all $r\in\Delta$. 
For every $r\in\Delta$ let $n(r)\ge 1$ be the least integer such that 
\begin{equation}\label{120180607aa}
\delta\lambda^{-(n(r)-1)}\le r^2.
\end{equation}
Let 
$$
R_r:=\bigcup\big\{R\in \mathcal R^{n(r)}:R\cap B(y,r)\ne\emptyset\big\}.
$$
Then item (b1) holds trivially and also $R_r$ is a set generated by the Markov partition $\mathcal R$ with 
$$
\ord(R_r)\le n(r).
$$
Invoking \eqref{120180607aa} and \eqref{320180614}, we see that
$$
R_r\subset B(y,r+r^2). 
$$
Because of this and \eqref{520180614aa} we have (b2). It follows from the definition of $n(r)$ that 
$$
\delta\lambda^{-(n(r)-2)}> r^2.
$$
Taking logarithms, yields
$$
n(r)< \frac{\log(\delta\lambda^2)-2\log r}{\log\lambda}.
$$
Therefore, using also \eqref{420180614}, we get that
$$
\begin{aligned}
0\le \limsup_{r\to 0} \ord(R_r)\mu(B(y,r))
&\le \limsup_{r\to 0} n(r)\mu(B(y,r)) \\
&\le \frac{C}{\log\lambda}\limsup_{r\to 0}\Big(r^\alpha\big(\log(\delta\lambda^2)-2\log r\big)\Big)
=0.
\end{aligned}
$$
This means that (b3) holds and the proof is complete.
\end{proof}

The main result of this section, and one of the main theorems of this paper, is the following.

\begin{thm}\label{mainthm}
Assume that $(X,\rho)$ is a compact metric space and $T\colon X\to X$ is an open topologically transitive distance expanding map. If $\mu$ is a Gibbs/equilibrium state for a H\"older continuous potential, then for every $y\in X$ we have that
\begin{equation}\label{eq:restA}
\limsup_{r\to 0} \tau_{B(y,r)}(x)\cdot \mu(B(y,r)) = +\infty, \mbox{\quad for $\mu$--a.e. $x\in X$.}
\end{equation}
\end{thm}

\begin{proof}
Fix a point $y \in X$. To ease notation, we will denote the ball $B(y,r)$ by $B_r$ and the complement of any set $Z$ in $X$ by $Z^c$. 

Fix $M>0$ and a radius $r>0$ belonging to $\Delta$, the set produced in Lemma~\ref{l1mp3}. Define
\begin{equation}
A_r:=\left\{x\in X : \tau_{B_r}(x) > \frac{M}{\mu(B_r)}\right\}.
\end{equation}
In order to prove our result it suffices to show that for every $M>0$ sufficiently large and for any $\delta>0$ 
we may define a decreasing sequence $r_n$, $n=1, 2,\ldots,\Omega$ ($\Omega$ to be chosen later) 
such that 
\begin{equation}\label{eq:meas0}
	\mu\left(\bigcap_{n=0}^{\Omega}A_{r_n}^c\right)\leq \delta.
\end{equation}
The definition of $A_r$ coupled with the results on the set $R_r$, namely (b1) and (b2) from Lemma~\ref{l1mp3} leads to
\begin{equation}\label{eq:Arcest}
\begin{aligned}
	\bigcap_{n=0}^{\Omega}A_{r_n}^c
&=\left\{x\in X : \forall_{0\leq i \leq \Omega}\;\tau_{B_{r_i}}(x)\leq \frac{M}{\mu({B_{r_i}})}\right\}\subset \left\{x\in X : \forall_{0\leq i \leq \Omega}\;\tau_{R_{r_i}}(x)\leq \frac{M}{\mu({R_{r_i}})}\frac{\mu({R_{r_i}})}{\mu({B_{r_i}})}\right\}\\
	&\subset \left\{x\in X : \forall_{0\leq i \leq \Omega}\;\tau_{R_{r_i}}(x)\leq \frac{2M}{\mu({R_{r_i}})}\right\},
	\end{aligned}
\end{equation}
where the last inequality holds assuming that $M>0$ is sufficiently large so that then $r_1>0$ is sufficiently small. Let us introduce the following notation. Given an integer $k\ge 1$ let
\begin{equation}
\label{eq:defqk}
	\begin{aligned}	
		a_r^{(k)}&:=\mu(\{x\in X : \exists_{1\le l\leq k}\; T^l{x}\in B_r\})=  \mu(\{x\in X : \tau_{B_r}(x)\leq k\}),\\
		q_r^{(k)}&:=\mu(\{x\in X : \exists_{1\le l\leq k}\; T^l{x}\in R_r\})= \mu(\{x\in X : \tau_{R_r}(x)\leq k\}).
	\end{aligned} 
\end{equation}
Trivially, $a_r^{(k)}\leq q^{(k)}_r$; and from the inclusions above 
\[ \mu(A_r^c)=a_r^{\left(\left[\frac{M}{\mu(B_r)}\right]\right)}\leq q_r^{\left(\left[\frac{2M}{\mu(R_r)}\right]\right)}.
\] 
Since the measure $\mu$ is $T$--invariant we get the estimate 
\begin{equation}\label{eq:estqk}
	q_r^{(k)} \leq k\mu(R_r). 
\end{equation}
Additionally, also because the measure $\mu$ is $T$--invariant, we get for all integers $0\le k_1\le k_2$ that
\begin{equation}\label{eq:measpres}
	\mu(\{x\in X : \exists_{k_1 \leq l \leq k_2}\; T^l{x}\in R_r\})= 
	\mu(\{x\in X : \exists_{0 \leq l \leq k_2-k_1}\; T^l{x}\in R_r \}).
\end{equation}
Now, we need a well-known upper bound of the measure of the intersection of two Markov sets. 
 Indeed, there exists a constant $C\in[1,+\infty)$ such that if $U$ and $V$ are   arbitrary sets generated be the Markov partition $\mathcal R$, then for every integer $k\ge \ord(U)$ we have that
\begin{equation}\label{eq:cor1}
	\mu(U\cap T^{-k}(V)) \leq C\mu (U)\mu(V).
\end{equation}
Put
$$
o(r):=\ord(R_r).
$$
Then 
$$
\ord\left(\bigcap_{i=0}^{k-o(r)}T^{-i}(R_r^c)\right)\le k
$$ 
and we have
\begin{equation}\label{eq:cor2}
\begin{aligned}	
\mu\left(T^{-(k+1)}(R_r) \cap \bigcap_{i=0}^{k-o(r)}T^{-i}(R_r^c)\right) 
&\leq C\mu \left(T^{-(k+1)}(R_r)\right)\cdot\mu\left(\bigcap_{i=0}^{k-o(r)}T^{-i}(R_r^c)\right)\\
&=C\mu \left(R_r\right)\cdot\mu\left(\bigcap_{i=0}^{k-o(r)}T^{-i}(R_r^c)\right),
\end{aligned}	
\end{equation}
where the inequality comes from \eqref{eq:cor1} and the equality from $T$ being measure-preserving. 

Using the estimate above we may find a satisfactory estimate on $q_r^{(k)}$. Observe that
\begin{equation}\label{eq:qrkfest}
\begin{aligned}
		q_r^{(k+1)} 
		&= q_r^{(k)} + \mu(x \in X : \tau_{R_r}(x)=k+1)\\
		&= q_r^{(k)} + 	\mu(x \in X : T^{k+1}(x)\in R_r \wedge \forall_{i\leq k} T^{i}(x)\notin R_r)\\
		&\leq q_r^{(k)} + 	\mu(x \in X : T^{k+1}(x)\in R_r \wedge \forall_{i\leq k-o(r)} T^{i}(x)\notin R_r)\\
		&= q_r^{(k)} + \mu\left(T^{-(k+1)}(R_r) \cap \bigcap_{i=0}^{k-o(r)} T^{-i}(R_r^c)\right)\\
		&\leq q_r^{(k)} + C\mu \left(R_r\right)\cdot\mu\left(\bigcap_{i=0}^{k-o(r)}T^{-i}(R_r^c)\right) 
		= q_r^{(k)} + C\mu \left(R_r\right)(1-q_r^{(k-o(r))}).
\end{aligned}
\end{equation}
The first line trivially yields  an estimate $q_r^{(k+1)} \leq q_r^{(k)}+\mu(R_r)$ and using this $o(r)$ times, yields
\begin{equation}
	q_r^{(k)}\leq q_r^{(k-o(r))}+o(r)\mu(R_r).
\end{equation}
Observe that by \eqref{eq:estqk} this inequality also holds if $k\leq o(r)$. We put this back into \eqref{eq:qrkfest}, arriving at
\begin{equation}
q_r^{(k+1)} 
\leq q_r^{(k)} + C\mu \left(R_r\right)(1-q_r^{(k)}+o(r)\mu(R_r)) 
= q_r^{(k)}(1-C\mu(R_r)) + C\mu(R_r)+ Co(r)\mu(R_r)^2.
\end{equation}
We apply this inductively $k$ times, and by observing that $q_r^{(1)}=\mu(R_r)$ we arrive at  
\begin{align*}
q_r^{(k)}
&\leq (1-C\mu(R_r))^{k-1}q_r^{(1)}+C\mu(R_r)\big(1+o(r)\mu(R_r)\big) \sum_{i=0}^{k-2} (1-C\mu(R_r))^i\\
&= (1-C\mu(R_r))^{k-1}\mu(R_r) + C\mu(R_r)(1+o(r)\mu(R_r)) \frac{1-(1-C\mu(R_r))^{k-1}}{C\mu(R_r)}.
\end{align*}
Let us return to the task of estimating $\mu(A_r^c)$. 
Applying the estimate on $q_r^{(k)}$ and Lemma~\ref{l1mp3}, making also a trivial simplification, and using common estimates on $e^x$, we get for all $r\in\Delta$ small enough, that
\begin{align*}
		q_r^{\left(\left[\frac{2M}{\mu(R_r)}\right]\right)}
&\leq (1-C\mu(R_r))^{\frac{2M}{\mu(R_r)}-2}\mu(R_r) + (1+o(r)\mu(R_r))(1-(1-C\mu(R_r))^{\frac{2M}{\mu(R_r)}-2})\\
&\leq e^{-2MC}(1-C\mu(R_r))^{-2}\mu(R_r)+ 
 (1+o(r)\mu(R_r))\left(1-e^{-3MC}(1-C\mu(R_r))^{-2}\right) \\
&\le \Gamma(M):=1-e^{-4MC}
< 1.
\end{align*}
Observe that this estimate also gives
\begin{equation}\label{eq:est-qk}
	k\leq \left[\frac{2M}{\mu(R_r)}\right]\implies q_r^k \leq \Gamma(M).
\end{equation}
We will additionally use a stronger mixing result, namely Theorem 5.4.10. from \cite{PUbook}, which applied to our situation means that
\begin{equation}\label{eq:estint}
		\mu\left(T^{-n}(A) \cap B\right) \leq (1+D\gamma^{n-k})\mu(A)\mu(B),
\end{equation}
where $D>0$ and $\gamma<1$ are some constants, $A$ is an arbitrary measurable set and $B\in\mathcal{F}_0^k$. Find $s\in\N$ such that 
\begin{equation}
	\Gamma(M)\cdot(1+D\gamma^s)<1.
\end{equation}
Denote $W:=1+D\gamma^s$.
We are finally ready to show that the required measure of the intersections is small, i.e. to prove \eqref{eq:meas0}.

For brevity, denote 
$$
R_i:=R_{r_i}, \  \  k_i:=\frac{2M}{\mu({R_{r_i}})},
\  \  {\rm and} \  \  
\tau_i:=\tau_{R_{r_i}}(x).
$$
Because of Lemma~\ref{l1mp3} there exists a decreasing sequence $(r_i)_{i=1}^\infty$ of positive radii, all belonging to $\Delta$, such that
\begin{align}
\label{eq:asom} \mbox{$\Omega$ may be taken so big that \quad}&(W\Gamma)^{\Omega+1}\leq \frac{\delta}{2},\\
\label{eq:asrii} \mbox{$r_i$ decrease so fast that \quad}	&k_{i+1}\geq 2(s+k_{i}), \mbox{\quad and }\\
	&\mu(R_{i+1}) \leq \frac{\delta}{2\Omega}
	\frac{1}{k_i+s}
	\label{eq:asri}
\end{align}
Using \eqref{eq:Arcest} first, and then dividing the set into $2^\Omega$ subsets depending on the behaviour of $\tau_i$ gives the following.
\begin{equation}\label{watpliwos}
\begin{aligned}
	\mu\Big(\bigcap_{n=0}^{\Omega}A_{r_n}^c\Big) &\leq \mu\left(\left\{\forall_{0\leq i \leq \Omega}\;\tau_i\leq k_i\right\}\right) \\
	&=	\mu\left(\left\{\tau_0\leq k_0 \wedge \forall_{1\leq i \leq \Omega}\;\left(\tau_i\leq k_{i-1}+s \vee \exists_{k_{i-1}+s < u \leq k_i} T^ux\in R_i
	\right)\right\}\right)\\
	&\leq\mu\left(\left\{\tau_0\leq k_0 \wedge \forall_{1\leq i \leq \Omega}\exists_{k_{i-1}+s < u \leq k_i} T^ux\in R_i\right\}\right) + \mu\left(\exists_{1\leq i \leq \Omega}\left\{\tau_i <k_{i-1}+s\right\}\right)\\
	&\leq \mu\left(\left\{\tau_0\leq k_0 \wedge \forall_{1\leq i \leq \Omega}\exists_{k_{i-1}+s < u \leq k_i} T^ux\in R_i\right\}\right) + \Omega\max_{1\leq i \leq \Omega}\mu\left(\left\{\tau_i <k_{i-1}+s\right\}\right)
\end{aligned}
\end{equation}
A series of estimates follows below. For the second summand, use the easy estimate on the entry time \eqref{eq:estqk} and then apply \eqref{eq:asri}. For the first summand, apply the estimate on the intersection \eqref{eq:estint} $\Omega$ times, use the definitions of $\Gamma(M)$ and $W$, then use the estimate \eqref{eq:est-qk} on $q_r^{(k)}$ along with \eqref{eq:measpres}, and finally invoke \eqref{eq:asom}. Using symbols:
\begin{align*}
	\mu\bigg(&\bigcap_{n=0}^{\Omega}A_{r_n}^c\bigg)\le \\
&\le \Omega \max_{1\leq i \leq \Omega} (k_{i-1}+s)\mu(R_i)+ (1+D\gamma^s)\mu\left(\left\{\tau_0\leq k_0\right\}\right) \mu\left(\left\{\forall_{1\leq i \leq \Omega}\exists_{k_{i-1}+s < u \leq k_i} T^ux\in R_i \right\}\right)\\
	&\leq \frac{\delta}{2} + W \Gamma(M) \mu\left(\left\{\forall_{1\leq i \leq \Omega}\exists_{k_{i-1}+s < u \leq k_i} T^ux\in R_i \right\}\right)\\
	&\leq \frac{\delta}{2} + W^{\Omega+1} \Gamma(M) \prod_{i=1}^{\Omega}\mu\left(\left\{\exists_{k_{i-1}+s < u \leq k_i} T^ux\in R_i \right\}\right)\leq \frac{\delta}{2} + W^{\Omega+1} \Gamma(M) \prod_{i=1}^{\Omega}\mu\left(\left\{\tau_i\leq k_i\right\}\right)\\
	&\leq \frac{\delta}{2} + W^{\Omega+1} \Gamma(M)^{\Omega+1} \leq \delta.
\end{align*}
 This ends the proof.
\end{proof}

\begin{rem}\label{r120181211}
What we have actually proved is that if $T\colon X\to X$ is an open topologically transitive distance expanding map of a compact metric space $(X,\rho)$ and $\mu$ is a Borel probability $T$--invariant measure on $X$ such that \eqref{eq:estint} holds and \eqref{420180614} holds for some point $y\in X$, then \eqref{eq:restA} holds for $\mu$--a.e. $x\in X$ with that point $y$. We also note that compactness of the metric space $X$ was not essential for the proof.
\end{rem}

\section{Thermodynamic Formalism of Subshifts of Finite Type with Countable Alphabet; Preliminaries}\label{sec:subshifts}

In this section we introduce the basic symbolic setting in which we will be working in the sequel. We will describe some fundamental thermodynamic concepts, ideas and results, particularly those used in later sections for applications.

Let $\mathbb{N}=\{1, 2, \ldots \}$  
and let $E$ be a countable  set, either finite or infinite,
called in the sequel an alphabet. Let
$$
\sg\colon E^\mathbb{N} \to E^\mathbb{N}
$$
be the  shift map\index{shift map}. 
It is given by the
formula
$$
\sg\( (\om_n)^\infty_{n=1}  \) =  \( (\om_{n+1})^\infty_{n=1}  \).
$$
We also put
$$
E^*=\bigcup_{n=0}^\infty E^n,
$$
to be the set of finite strings.
\fr For every $\om \in E^*$, we denote by $|\om|$  the unique integer
$n \geq 0$ such that $\om \in E^n$. We  call $|\om|$ the length of
$\om$. We make the convention that $E^0=\{\es\}$. If $\om \in
E^\mathbb{N}$ and $n \geq 1$, we put
$$
\om |_n=\om_1\ldots \om_n\in E^n.
$$
If $\tau \in E^*$ and $\om \in E^* \cup E^\mathbb{N}$, we define the concatenation of $\tau$ and $\omega$ by:  
$$
\tau\om:=
\begin{cases}
\tau_1\dots\tau_{|\tau|}\om_1\om_2\dots\om_{|\om|} \  \  &{\rm if } \ \om \in E^*, \\
\tau_,\dots\tau_{|\tau|}\om_1\om_2\dots \  \  &{\rm if } \ \om \in E^\N,.
\end{cases}
$$
Given $\om,\tau\in E^{\mathbb N}$, we define $\omega\wedge\tau  \in
E^{\mathbb N}\cup E^*$ to be the longest
initial block common to both $\om$ and $\tau$. For each $\alpha >
0$, we define a metric $d_\alpha$ on
$E^{\mathbb N}$ by setting
\begin{equation}\label{d-alpha}
d_\alpha(\om,\tau) ={\rm e}^{-\alpha|\om\wedge\tau|}.
\end{equation}
All these metrics induce the same topology, known to be the product (Tichonov) topology. A real or complex valued function defined on a subset of $E^\N$ is 
H\"older with respect to one of these metrics if and only if it is
H\"older with respect to all of them, although, of course, the H\"older exponent depends
on the metric. If no metric is specifically mentioned, we take it to
be $d_1$.

Now consider an arbitrary matrix $A\colon E \times E \to \{0,1\}$. Such a matrix will be called the incidence matrix in the sequel. Set
$$
E^\infty_A
:=\{\om \in E^\mathbb{N}:  \,\, A_{\om_i\om_{i+1}}=1  \,\, \mbox{for
  all}\,\,   i \in \N
\}.
$$
Elements \index{A-admissible matrices@$A$-admissible matrices} of $E^\infty_A$ are called {\it $A$-admissible}. We also set
$$
E^n_A
:=\{\om \in E^\mathbb{N}:  \,\, A_{\om_i\om_{i+1}}=1  \,\, \mbox{for
  all}\,\,  1\leq i \leq
n-1\}, \ n \in \N,
$$
and
$$
E^*_A:=\bigcup_{n=0}^\infty E^n_A.
$$
The elements of these sets are also called {\it $A$-admissible}. For
every  $\om \in E^*_A$, we put
$$
[\om]:=\{\tau \in E^\infty_A:\,\, \tau_{|_{|\om|}}=\om \}.
$$
The set $[\om]$ is called the cylinder generated by the word $\om$. The collection of all such cylinders forms a base for the product topology relative to $E^\infty_A$.
The following fact is obvious.

\begin{prop}\label{p1j83}
The set $E^\infty_A$ is  a closed subset of
$E^\mathbb{N}$, invariant under the shift map $\sg\colon  E^\mathbb{N}\to E^\mathbb{N}$, the latter meaning that
$$
\sg(E^\infty_A)\sbt E^\infty_A.
$$
\end{prop}

The matrix $A$ is said to be {\it finitely irreducible} if there
exists a finite set $\Lambda \sbt E_A^*$ such that for all $i,j\in E$
there exists $\om\in \Lambda$ for which $i\om j\in E_A^*$. If all elements of some such $\Lambda$ are of the same length, then $A$ is called finitely primitive (or aperiodic).

\medskip The topological pressure of a continuous function $f:E_A^\infty\to\R$ with respect to the shift map $\sg:E_A^\infty\to E_A^\infty$ is defined to be
\begin{equation}\label{2.1.1}
\P(f)
:=\lim_{n\to\infty}\frac{1}{n}\log \sum_{\om \in E_A^n}\exp \left(\sup_{\tau\in [\om]}\sum_{j=0}^{n-1}f(\sg^j(\tau))\right).
\end{equation}
The existence of this limit, following from the observation that the ``$\log$'' above forms a subadditive sequence, was established in \cite{MU_Israel}, comp. \cite{MU_GDMS}. Following the common usage we abbreviate 
$$
S_nf:=\sum_{j=0}^{n-1}f\circ\sg^j
$$
and call $S_nf(\tau)$ the $n$th Birkhoff's sum of $f$ evaluated at a word $\tau\in E_A^\infty$.

\medskip A function $f\colon E_A^\infty\lra\R$ is called (locally) H\"older continuous with an exponent $\a>0$ if 
$$
V_\a(f):=\sup_{n\ge 1}\left\{V_{\a,n}(f)\right\}<+\infty,
$$
where 
$$
V_{\a,n}(f)=\sup\{|f(\om)-f(\tau)|e^{\a(n-1)}:\om,\tau\in E_A^\infty
\text{ and } |\om\wedge \tau|\ge n\}.
$$
A function $f\colon E_A^\infty\to\R$ is called summable if  
\begin{equation}\label{2.3.1}
\sum_{e\in E}\exp(\sup(f|_{[e]}))<+\infty. 
\end{equation}
We note that if $f$ has a Gibbs state, then $f$ is summable.
The following theorem has been proved in \cite{MU_Israel}, comp. \cite{MU_GDMS}, for the class of acceptable functions defined there. Since H\"older continuous ones are among them, we have the following.

\begin{thm}[Variational Principle]\label{t2.1.6}
If the incidence matrix $A\colon E\times E\to\{0,1\}$ is 
finitely irreducible and if $f\colon E_A^\infty\to\R$ is H\"older continuous, then
$$
\P(f)=\sup\Big\{\h_{\mu}(\sg)+\int f\,d\mu\Big\},
$$
where the supremum is taken over all $\sg$-invariant (ergodic) Borel
probability measures $\mu$ such that $\int f\,d\mu >-\infty$.
\end{thm}

We call a $\sg$-invariant probability measure $\mu$ on $E_A^\infty$ an  equilibrium state of a H\"older continuous function $f\colon E_A^\infty\to\R$ if $\int -f\,d\mu<+\infty$ and 
\begin{equation}\label{2.2.9}
\h_{\mu}(\sg)+\int \!\! f\, d\mu=\P(f). 
\end{equation}
It was proved in \cite{MU_Israel} 
that if the matrix  $A$ is finitely irredicible and $f\colon E_A^\infty\to\R$ is a H\"older continuous summable potential, then there exists a unique equilibrium state of $f$. We denote it by $\mu_f$. For the reasons explained in  \cite{MU_Israel} and \cite{MU_GDMS} it is also called a Gibbs state for $f$ and, crucially, it satisfies \eqref{eq:estint}. We have the following. 

\begin{thm}\label{t120181213}
Let $E$ be a countable set, either finite or infinite, and let $A\colon E \times E \to \{0,1\}$ be a finitely irreducible incidence matrix. If $f\colon E_A^\infty\to\R$ is a H\"older continuous summable potential, then, given any $\alpha>0$, we have for the dynamical system $\(\sg\colon E^\infty_A\to E^\infty_A,\mu_f\)$ that
\begin{equation}\label{120181213}
\limsup_{r\to 0} \tau_{B_\alpha(\rho,r)}(\om)\cdot \mu_f(B_\alpha(\rho,r)) = +\infty, 
\end{equation}
for every $\rho\in E^\infty_A$ and $\mu_f$--a.e. $\omega\in E^\infty_A$, where $B_\alpha(\gamma,r)$ denotes the ball, with respect to the metric $d_\alpha$ defined in \eqref{d-alpha}, centred at $\gamma\in E^\infty_A$ with radius $r$. 
\end{thm}

If $E$ is a finite set, then this theorem is an immediate consequence of Theorem~\ref{mainthm} once one observes that then $\sg\colon E^\infty_A\to E^\infty_A$ is an open topologically transitive distance expanding map on the compact space $E^\infty_A$ with respect to every metric $d_\alpha$, $\alpha>0$. In order to see that this theorem holds in its full generality, i. e. for all countable sets $E$, it suffices to note the following: 

\begin{enumerate}
\,

\item Given $\alpha>0$, for every $r\in(0,1)$ there exists an integer $n_r\asymp-\frac1\alpha\log r\ge 0$ such that 
$$
B_\alpha(\tau,r)=\big[\tau|_{n_r}\big]
$$
for every $\tau\in E^\infty_A$ and 
$$
\lim_{r\to 0}n_r=+\infty.
$$
\item There exists $\beta>0$ (see \cite{MU_GDMS}) such that 
$$
\mu_f\big([\tau|_n]\big)\le e^{-\beta n}
$$
for every $\tau\in E^\infty_A$ and every integer $n\ge 0$.

\,

\item Because of (1), (2), and Remark~\ref{r120181211}, the proof of Theorem~\ref{mainthm} goes through in the current setting. 
\end{enumerate}

\medskip\noindent In fact, as an immediate consequence of Theorem~\ref{t120181213} and items (1) and (2) above, we get the following version of this theorem, independent of any metric $d_\alpha$:

\begin{thm}\label{t220181213}
Let $E$ be a countable set, either finite or infinite, and let $A\colon E \times E \to \{0,1\}$ be a finitely irreducible incidence matrix. If $f\colon E_A^\infty\to\R$ is a H\"older continuous summable potential, then we have for the dynamical system $\(\sg\colon E^\infty_A\to E^\infty_A,\mu_f\)$ that
\begin{equation}\label{120181213new}
\limsup_{n\to\infty}\tau_{[\rho|_n]}(\om)\cdot \mu_f([\rho|_n]) = +\infty, 
\end{equation}
for every $\rho\in E^\infty_A$ and $\mu_f$--a.e. $\omega\in E^\infty_A$.
\end{thm}

\section{Graph Directed Markov Systems}\label{GDMS}

Our goal now is to go beyond (uniformly) expanding maps. It will be accomplished in two major steps. First, in the current section, we will extend the previous results to some class of maps of infinite degree. These will be the maps naturally resulting from graph directed Markov systems with a countable infinite alphabet. The second major step, carried on in the next section, will be to employ the techniques of the first return maps. We will first relate, see Lemma~\ref{lem:insys}, hitting times for a given system and an induced one. Then, as a consequence, we will be able to prove an analog of Theorem~\ref{mainthm} for all systems that allow inducings having structure of graph directed Markov systems. Finally, in Section~\ref{Examples}, we will provide several classes of examples.

\medskip We now define a graph directed Markov system (\emph{abbr.} GDMS) relative to a directed multigraph $(V,E,i,t)$ with an incidence matrix $A$. A \emph{directed multigraph} consists of 
\begin{itemize} 

\, 

\item a finite set $V$ of vertices,

\,

\item a countable (either finite or infinite) set $E$ of directed edges,

\,

\item a map $A\colon E\times E\to \{0,1\}$ called an \emph{incidence matrix} on $(V,E)$,

\,

\item two functions $i,t\colon E\to V$, such that $A_{ab} = 1$ implies $t(b) = i(a)$. 
\end{itemize}
In addition, we have a collection of non--empty compact metric spaces $\{X_v\}_{v\in V}$ and a number $s\in (0,1)$, and for every $e\in E$, we have a 1-to-1 contraction $\phi_e\colon X_{t(e)}\to X_{i(e)}$ with Lipschitz constant $\le s$. Then the collection
\[
\cS := \{\phi_e\colon X_{t(e)}\lra X_{i(e)}\}_{e\in E}
\]
is called a GDMS. We now describe the limit set of the system $\cS$. 
For each $n \geq 1$ and $\omega \in E_A^n$, we consider the map coded
by $\omega$ 
\[\phi_\omega:=\phi_{\omega_1}\circ\cdots\circ\phi_{\omega_n}\colon X_{t(\omega)}\lra X_{i(\omega)}.
\]
For $\omega \in E^\infty_A$, the sets $\{\phi_{\omega|_n}\left(X_{t(\omega_n)}\right)\}_{n \geq 1}$ form a descending sequence of non-empty compact sets and therefore $\bigcap_{n \geq 1}\phi_{\omega|_n}\left(X_{t(\omega_n)}\right)\ne\emptyset$. Since for every $n \geq 1$, 
\[
\diam\left(\phi_{\omega|_n}\left(X_{t(\omega_n)}\right)\right)\le s^n\diam\left(X_{t(\omega_n)}\right)\le s^n\max\{\diam(X_v):v\in V\},
\]
we conclude that the intersection 
$$
\bigcap_{n \in \N}\phi_{\omega|_n}\left(X_{t(\omega_n)}\right)
$$ 
is a singleton and we denote its only element by $\pi(\omega)$. In this 
way we have defined the map 
\[
\pi\colon E^\infty_A\lra X:=\du_{v\in V}X_v
\]
from $E_A^\infty$ to $\du_{v\in V}X_v$, the disjoint union of the compact
sets $X_v$. The set 
\[
J=J_\cS=\pi(E^\infty_A)
\]
will be called the limit set of the GDMS $S$. 

A GDMS is called an iterated function system (\emph{abbr.} IFS)  if $V$, the set of vertices, is a singleton and the incidence matrix $A$ consists of $1$s only, i.e. $A(E\times E)=\{1\}$. 

\begin{dfn}\label{definitionsymbolirred}
We call the GDMS $\cS$ and its incidence matrix $A$ \emph{finitely  irreducible} if there exists a finite set $\Lambda\subset E_A^*$ such that for all $a, b\in E$ there exists a word $\omega\in\Lambda$ such that the concatenation $a\omega b$ is in $E_A^*$.  
\end{dfn}

Given an integer $n\ge 1$ and a set $F\sbt E_A^n$, we call the set
\[
U:=\bu_{\om\in F}\phi_\om\(X_{t(\om)}\)
\]
a set of order $\le n$ generated by the GDMS $S$. If in addition $U$ cannot be represented as a union of sets of the form $\phi_\tau\(X_{t(\tau)}\)$, $\tau\in E_A^{n-1}$, then $R$ will be called of order $n$, and we will write
$$
n=\ord(U).
$$

Assume now that for some integer $d\ge 1$, $X_v$ is a subset of $\R^d$ for every vertex $v\in V$. Assume further that 
$$
\ov{\Int(X_v)}=X_v
$$
and 
\begin{equation}\label{120180815}
\phi_a\(\Int\(X_{t(a)}\)\)\cap \phi_b\(\Int\(X_{t(b)}\)\)=\es
\end{equation}
whenever $a, b\in E$ and $a\ne b$. This assumption is commonly called the Open Set Condition (\emph{abbr.} OSC).

A 
GDMS $S=\{\phi_e\}_{e\in E}$ is said to satisfy the Strong Open Set Condition (\emph{abbr.} SOSC) if it satisfies the OSC and 
\[
J_S\cap \Int X\ne\es.
\]

We want to define an ordinary dynamical system out of the GDMS $S$. The problem is that the map projection map $\pi \colon E^\infty_A\longrightarrow X$ need not be $1$-to-$1$. In order to remedy this  problem (i.e. with non-unique coding), we introduce the set
\[
\mathring{J}_\cS:=J_\cS\sms\bu_{\om\in E_A^*}\phi_\om(\partial X_{t(\om)}).
\]
Set $\mathring{E}_A^\infty:=\pi_\cS^{-1}\(\mathring{J}\)$
and notice that for every $z\in\mathring{J}_\cS$ there exists a unique $\om(z)\in E_A^\infty$ such that
$z=\pi(\om(z))$.
Moreover, $\om(z)\in \mathring{E}_A^\infty$ and we simply denote it by $\pi^{-1}(z)$. Note that 
$$
\sg\(\mathring{E}_A^\infty\)\sbt \mathring{E}_A^\infty
$$
and this restricted shift map induces a map 
$$
T_\cS\colon \mathring{J}_\cS\lra\mathring{J}_\cS
$$ 
by the formula
\[
T_\cS(z)=\pi\circ\sg(\pi^{-1}(z))\in\mathring{J}_\cS,
\]
so that the following diagram commutes 
\[\begin{CD}
\mathring{E}_A^\infty @>\sigma>> \mathring{E}_A^\infty\\
@V\pi VV @VV\pi V\\
\mathring{J}_\cS @>>T_\cS> \mathring{J}_\cS
\end{CD}
\]
and the map $\pi\colon\mathring{E}_A^\infty\lra\mathring{J}_\cS$ is a continuous bijection. 
The proof of the following result can be found in \cite{KU-NCPII}.
\begin{thm}\label{t2sc3}
If $S=\{\phi_e\}_{e\in E}$ is a GDMS satisfying {\rm SOSC} and $\mu$ is a Borel probability $\sg$\nobreakdash--invariant ergodic measure on $E_A^\infty$ with full topological support, then
\begin{equation}
\mu\circ\pi^{-1}(\mathring{J}_\cS)=1
\end{equation}
and 
\begin{equation}
\mu\circ\pi^{-1}\circ  T_\cS^{-1}=\mu\circ\pi^{-1}.
\end{equation}
\end{thm}
If $f\colon E_A^\infty\lra\R$ is a H\"older continuous summable potential, then denote 
$$
\hat\mu_f:=\mu_f\circ\pi^{-1}.
$$
As an immediate consequence of Theorem~\ref{t2sc3} and bijectivity of the map $\pi\colon\mathring{E}_A^\infty\to\mathring{J}_\cS$, we get the following.

\begin{cor}\label{o2_2016_0211}
Let $\cS$ be a finitely irreducible GDMS satisfying the SOSC. Let $f\colon E_A^\infty \lra\R$ be a H\"older continuous summable potential. Denote by $\mu_f$ its unique $\sg$\nobreakdash--invariant Gibbs/equilibrium state. Then 
\[
\mu_f\(\mathring{E}_A^\infty\)=1 
\  \ {\rm and} \  \
\hat\mu_f\(\mathring{J}_\cS\)=1.
\]
Moreover, the projection $\pi\colon \mathring{E}_A^\infty\lra\mathring{J}_\cS$ establishes a measure--preserving isomorphism between measure--preserving dynamical systems $\(\sg\colon\mathring{E}_A^\infty\to\mathring{E}_A^\infty, \mu_f\)$ and $\(T_\cS\colon\mathring{J}_\cS\to\mathring{J}_\cS,\hat\mu_f\)$. 
\end{cor}

The following lemma is entirely analogous to Lemma~\ref{l1mp3}. The proof is also analogous. Since however, the proof is short, we provide it here for the sake of completeness and convenience of the reader.

\begin{lem}\label{l1_2018_08_02}
Assume that $S$ is GDMS and all spaces $X_v$, $v\in V$, are compact subsets of some (finitely dimensional) Euclidean space. Assume also that $\mu$ is a Borel probability measure on the limit set $J_S$. If $y\in J_S$ and 
\begin{equation}\label{420180614Bb}
\mu(B(y,r))\le Cr^\alpha
\end{equation}
for some constant $C\ge 1$, $\alpha>0$, and all radii $r>0$, then there exists a Lebesgue measurable set $\Delta\subset (0,1)$ with the following properties
\begin{enumerate}[(a)]
	
\item $\hfill \displaystyle \lim_{\Delta\ni r\to 0}\frac{Leb(\Delta\cap(0,r))}{r}=1. \hfill$\vskip0.5em
\item 
For every $r\in\Delta$ there exists a set $R_r$, generated by $S$, satisfying\vskip0.5em
\begin{enumerate}[(b1)]
\item $\hfill \displaystyle B(y,r) \subset R_r, \hfill$ \vskip0.5em
\item $\hfill \displaystyle \frac{\mu(R_r)}{\mu(B(y,r))}\le 2,\hfill$ \vskip0.5em
\item $\hfill \displaystyle \lim_{r\to 0} \ord(R_r)\mu(B(y,r))=0. \hfill$
\end{enumerate}
\end{enumerate}
\end{lem}

\begin{proof}
Lemma~3.6 from \cite{PUZI}, applied with $\kappa_y=2$, gives a Lebesgue measurable set 
$\Delta\subset (0,1)$ for which item (a) above holds and 
\begin{equation}\label{520180614}
\frac{\mu(B(y,r+r^2))}{\mu(B(y,r))}\le 2
\end{equation}
for all $r\in\Delta$. 
Let $n(r)\ge 1$ be the least integer such that 
\begin{equation}\label{120180607}
\delta s^{(n(r)-1)}\le r^2.
\end{equation}
Let 
$$
R_r:=\bigcup\big\{\phi_\om(X_{t(\om)}\):\om\in E_A^{n(r)} \  \  {\rm and}\ \ B(y,r)\cap \phi_\om\(X_{t(\om)}\)\ne\emptyset\big\}.
$$
Item (b1) holds trivially, and $R_r$ is a set generated by the Markov partition $\mathcal R$ with 
$$
\ord(R_r)\le n(r).
$$
Invoking \eqref{120180607} and \eqref{320180614}, we see that
$
R_r\subset B(y,r+r^2). 
$
This and \eqref{520180614} yields (b2). It follows from the definition of $n(r)$ that 
$$
\delta s^{(n(r)-2)}> r^2.
$$
Taking logarithms gives
$$
n(r)< \frac{\log(\delta s^{-2})-2\log r}{-\log s}
$$
Combining all that with \eqref{420180614Bb}, we get that
$$
\begin{aligned}
0\le \limsup_{r\to 0} \ord(R_r)\mu(B(y,r))
&\le \limsup_{r\to 0} n(r)\mu(B(y,r)) \\
&\le \frac{C}{\log(1/s)}\limsup_{r\to 0}\Big(r^\alpha\big(\log(\delta s^{-2})-2\log r\big)\Big)
=0,
\end{aligned}
$$
thus (b3) holds and the proof is complete.
\end{proof}

\noindent Having this lemma, Corollary~\ref{o2_2016_0211}, and already knowing that \eqref{eq:estint} holds for all equilibrium states of H\"older continuous summable potentials, the same proof (see Remark~\ref{r120181211}) as the one of Theorem~\ref{mainthm}, gives the following.

\begin{thm}\label{mainthm_2}
Let $\cS$ be a finitely irreducible GDMS satisfying the SOSC. Let $f\colon E_A^\infty\lra\R$ be a H\"older continuous summable potential. If $y\in \mathring{J}_\cS$ and there exist constants $\alpha>0$ and $C\ge 1$ such that
\begin{equation}\label{420180614B}
\hat\mu_f(B(y,r))\le Cr^\alpha
\end{equation}
for all radii $r>0$, then for the dynamical system $\(T_\cS\colon\mathring{J}_\cS\to\mathring{J}_\cS,\hat\mu_f\)$ we have that
\begin{equation}\label{eq:rest2}
\limsup_{r\to 0} \tau_{B(y,r)}(x)\cdot \hat\mu_f(B(y,r)) = +\infty, \mbox{\quad for $\hat\mu_f$--a.e. $x\in\mathring{J}_\cS$.}
\end{equation}
\end{thm}

\noindent In order to meaningfully apply this theorem, we shall prove the following.

\begin{prop}\label{p1sam1}
Let $\cS$ be be a finitely irreducible conformal GDMS satisfying the SOSC. If $f\colon E_A^\infty\lra\R$ is a H\"older continuous summable potential, then there exists $\Lambda$, a Borel subset of $E_A^\infty$ such that $\mu_f(\Lambda)=1$ and for every $\om\in\Lambda$ there exist $C\ge 1$ and $\alpha>0$ such that
$$
\mu_f(B(\pi(\om),r)\le Cr^\alpha
$$
for all radii $r\ge 0$.
\end{prop}

\begin{proof}
Since $\cS$ satisfies the SOSC and since the measure $\mu_f$ is of full topological support, there exists $v\in V$, $\xi\in X_v$, and $R\in (0,1)$ such that 
\begin{equation}\label{220181130}
B(\xi,4R)\sbt X_v
\end{equation}
and 
$$
\mu_f\big(\pi^{-1}(B(\xi,R))\big)=\hat\mu_f(B(\xi,R))>0.
$$
Let $\Lambda_1\sbt E_A^\infty$ be the set of all points $\omega$ such that
\begin{equation}\label{120181130}
\sg^n(\om)\in \pi^{-1}(B(\xi,R))
\end{equation}
for infinitely many positive $n$'s. By Birkhoff's Ergodic Theorem and ergodicity of measure $\mu_f$, we have that
$$
\mu_f(\Lambda_1)=1.
$$
Take any $\omega \in \Lambda_1$. Let $(n_k(\om))_{k=1}^\infty$ be the infinite strictly increasing sequence of positive integers for which \eqref{120181130} holds. Abbreviate
$
n_k:=n_k(\om). 
$
Then, because of \eqref{220181130}, we have for every $k\ge 1$ that
$$
B\big(\pi(\sg^{n_k}(\om)),2R\big)\sbt \Int(X_v)
$$
and for some fixed $K\geq 1$
$$
B\Big(\pi(\om),K^{-1}\Big|\phi_{\om|_{n_k}}'\big(\pi(\sg^{n_k}(\om))\big)\Big|R\Big)
\subset \phi_{\om|_{n_k}}B\big(\pi(\sg^{n_k}(\om)),R\big).
$$
$\mu_f$ is a Gibbs measure, allowing us to write, for some constant $Q\in[0,+\infty)$, that
$$
\begin{aligned}
\mu_f\Big(B\Big(\pi(\om),K^{-1}\big|\phi_{\om|_{n_k}}'\big(\pi(\sg^{n_k}(\om))\big)\big|R\Big)\Big)
&\le \mu_f\Big(\phi_{\om|_{n_k}}B\big(\pi(\sg^{n_k}(\om)),R\big)\Big) \\
&\le Q\exp\big(S_{n_k}f(\om)-\P(f)n_k\big).
\end{aligned}
$$
Now, given $r>0$, there exist a largest integer $k_r^-=k_r^-(\om)\ge 1$ and a smallest integer $k_r^+=k_r^+(\om)\ge 1$ such that respectively
\begin{equation}\label{1sam2}
K^{-1}\Big|\phi_{\om|_{n_{k_r^+}}}'\big(\pi(\sg^{n_{k_r^+}}(\om))\big)\Big|R
\le r
\  \  \  {\rm and}  \  \  \
r\le K^{-1}\Big|\phi_{\om|_{n_{k_r^-}}}'\big(\pi(\sg^{n_{k_r^-}}(\om))\big)\Big|R.
\end{equation}
Then 
$$
B(\pi(\om),r)
\subset B\Big(\pi(\om),K^{-1}\Big|\phi_{\om|_{n_{k_r^-}}}'\big(\pi(\sg^{n_{k_r^-}}(\om))\big)\Big|R\Big),
$$
and, consequently
$$
\mu_f(B(\pi(\om),r))
\le Q\exp\left(S_{n_{k_r^-}}f(\om)-\P(f)n_{k_r^-}\right).
$$
But, by Birkhoff's Ergodic Theorem, there exists a measurable set $\Lambda_2\subset\Lambda_1$ such that $\mu_f(\Lambda_2)=1$ and 
$$
\lim_{n\to\infty}\frac1nS_nf(\omega)=\mu_f(f):=\int f\,d\mu_f
$$
for every $\omega\in \Lambda_2$. Therefore, for every $\om\in \Lambda_2$ there exists $N_1(\om)\ge 1$ such that if $n\ge N_1(\om)$, then, using the fact that the entropy $\h_{\mu_f}(\sg)$ is positive, we get that
$$
\frac1nS_nf(\omega)\le \mu_f(f)+\frac12\h_{\mu_f}(\sg).
$$
Therefore, for all $\om\in \Lambda_2$ and all radii $r>0$ small enough, say $0<r\le r_\om$, we get that 
\begin{equation}\label{2sam2}
\mu_f(B(\pi(\om),r))
\le Q\exp\left(\Big(\mu_f(f)+\frac12\h_{\mu_f}(\sg)-\P(f)\Big)n_{k_r^-}\right)
=Q\exp\left(-\frac12\h_{\mu_f}(\sg)n_{k_r^-}\right),
\end{equation}
where writing the equality sign we used the fact that $\mu_f$ an (in fact unique) equilibrium state, 
i.e. 
$$
\P(f)=\h_{\mu_f}(\sg)+\mu_f(f).
$$
On the other hand, it follows from Birkhoff's Ergodic Theorem that 
there exists a measurable set $\Lambda_3\subset\Lambda_2$ such that $\mu_f(\Lambda_3)=1$ and 
$$
-\lim_{n\to\infty}\frac1n\log\big|\phi_{\tau|_n}'\big(\pi(\sg^n(\tau))\big)\big|
=\chi_{\mu_f}
:=-\int_{E_A^\infty}\log\big|\phi_{\rho_1}'\big(\pi(\sg(\rho))\big)\big|\,d\mu_f(\rho)>0
$$
for all $\tau\in \Lambda_3$. Hence, if $\tau\in \Lambda_3$, then for every integer $n\ge 0$ large enough,
$$
\big|\phi_{\tau|_n}'\big(\pi(\sg^n(\tau))\big)\big|
\ge \exp\Big(-\frac12\chi_{\mu_f}n\Big).
$$
In conjunction with \eqref{1sam2}, this implies that
$$
\exp\Big(-\frac12\chi_{\mu_f}n_{k_r^+}\Big)\le KR^{-1}r
$$
for all $r>0$ small enough. Combining this with \eqref{2sam2}, we obtain:
\begin{equation}\label{1sam3}
\begin{aligned}
\mu_f(B(\pi(\om),r))
&\le Q\exp\left(-\frac12\chi_{\mu_f}n_{k_r^+}\frac{\h_{\mu_f}(\sg)}{\chi_{\mu_f}}\frac{n_{k_r^-}}{n_{k_r^+}}\right)
=Q\left(\exp\bigg(-\frac12\chi_{\mu_f}n_{k_r^+}\bigg)\right)^{\frac{\h_{\mu_f}(\sg)}{\chi_{\mu_f}}\frac{n_{k_r^-}}{n_{k_r^+}}} \\
&\le Q(KR^{-1})^{\frac{\h_{\mu_f}(\sg)}{\chi_{\mu_f}}\frac{n_{k_r^-}}{n_{k_r^+}}}\, \, r^{\frac{\h_{\mu_f}(\sg)}{\chi_{\mu_f}}\frac{n_{k_r^-}}{n_{k_r^+}}} \\
&\le Q(KR^{-1})^{\frac{\h_{\mu_f}(\sg)}{\chi_{\mu_f}}}\, \, r^{\frac{\h_{\mu_f}(\sg)}{\chi_{\mu_f}}\frac{n_{k_r^-}}{n_{k_r^+}}},
\end{aligned}
\end{equation}
where we wrote the last inequality since $KR^{-1}\ge 1$ and 
$n_{k_r^-}/n_{k_r^+}\le 1$. Since $k_r^+-k_r^-\le 1$  and since, by Birkhoff's Ergodic Theorem, there exists a measurable set $\Lambda_4\subset\Lambda_3$ such that $\mu_f(\Lambda_4)=1$ and 
$$
\lim_{j\to\infty}\frac{n_{j+1}(\tau)}{n_j(\tau)}=1
$$
for all $\tau\in \Lambda_4$, we conclude that if $\om\in \Lambda_4$, then
$$
\frac{n_{k_r^+(\om)}(\om)}{n_{k_r^-(\om)}(\om)}\le 2
$$
for all $r>0$ small enough. Inserting this into \eqref{1sam3}, we get for all $r>0$ small enough that
$$
\mu_f(B(\pi(\om),r))\le Cr^\alpha,
$$
where $C=Q(KR^{-1})^{\frac{\h_{\mu_f}(\sg)}{\chi_{\mu_f}}}$ and $\alpha=\frac12\frac{\h_{\mu_f}(\sg)}{\chi_{\mu_f}}$. The proof of Proposition~\ref{p1sam1} is thus complete.
\end{proof}

As an immediate consequence of this proposition and Theorem~\ref{mainthm_2}, we get the following.

\begin{thm}\label{mainthm_20181121}
Let $\cS$ be a finitely irreducible conformal GDMS which satisfies the SOSC. If $f\colon E_A^\infty\lra\R$ is a H\"older continuous summable potential, then for the dynamical system $\(T_\cS\colon\mathring{J}_\cS\lra\mathring{J}_\cS,\hat\mu_f\)$ we have that
\begin{equation}\label{eq:rest3}
\limsup_{r\to 0} \tau_{B(y,r)}(x)\cdot \hat\mu_f(B(y,r)) = +\infty, 
\end{equation}
for $\hat\mu_f$--a.e. $y \in\mathring{J}_\cS$ and $\hat\mu_f$--a.e. $x\in\mathring{J}_\cS$.
\end{thm}

\section{First Return Map Techniques}\label{FRMT}

The following result states that the upper limit, the subject of this paper, does not change when we go to an induced system. As we have already announced in the previous section, this is our crucial step to apply the results of the previous section concerning graph directed Markov systems to many other classes of non--uniformly expanding systems. 

\begin{lem}\label{lem:insys}
Consider an ergodic, metric, measure preserving, dynamical system $(X,T,\mu,d)$; in particular $\mu$ is a Borel probability measure on $X$. Let  $\widehat{X}\subset X$ be a measurable subset of $X$ with positive measure $\mu$. Define (in the standard way) the first return map (induced system) 
$(\widehat{X}, \widehat{T}, \widehat{\mu},d)$. Assume that $y\in \Int \widehat{X}$. Then the upper limits of entry times, i.e. \eqref{eq:restA}, for the original and the induced system coincide. Likewise for lower limits. In other words,
\begin{equation}
\limsup_{r\to 0} \tau_{B(y,r)}(x)\cdot \mu(B(y,r))= \limsup_{r\to 0} \widehat{\tau}_{B(y,r)}(x)\cdot \widehat{\mu}(B(y,r))
\end{equation}
and 
\begin{equation}
\lim\inf_{r\to 0} \tau_{B(y,r)}(x)\cdot \mu(B(y,r))= \liminf_{r\to 0} \widehat{\tau}_{B(y,r)}(x)\cdot \widehat{\mu}(B(y,r))
\end{equation}
for $\widehat{\mu}$--a.e. $x\in\widehat{X}$, where $\widehat{\tau}_U$ is the first entry time into a set $U\sbt\widehat{X}$ for the induced system $(\widehat{X},\widehat{T})$.
\end{lem}
\begin{proof}
Having taken $x\in\widehat{X}$, we start by defining the sequence of subsequent \emph{closest} approaches to $y$. Inductively:
\begin{align*}
n_1&:=1, \\
n_{k+1}&:= \min\left\{n:d(T^nx,y)<d(T^{n_k}x,y)\right\}.
\end{align*}
Also, denote $r_k:= d(T^{n_k}x,y)$. Using this notation we have that
\[\tau_{B(y,r)}(x)= n_k \mbox{\quad if \quad} r_k<r \leq r_{k-1}. \]
For brevity, denote the first return map into $\widehat{X}$ by $t(x):=\tau_{\widehat{X}}(x)$, and define:
\[A_l(x):=\sum_{i=0}^{l-1}t(\widehat{T}^i(x)). \]
 By the definition of the induced map we have 
\begin{equation}\label{eq:ind1}
\widehat{T}^l(x) = T^{A_l(x)}(x),\end{equation}
and therefore, by ergodicity of $\mu$, along with Birkhoff's Ergodic Theorem and Kac's Lemma, we get that
\begin{equation}\label{eq:ind2}
\lim_{l\to \infty}\frac{1}{l}A_l(x) 
= \int_{\widehat{X}}t\,d\widehat{\mu}
=\frac{1}{\mu(\widehat{X})}
\end{equation}
for $\mu$--a.e. $x\in X$, say $x\in Z\sbt X$, where $Z$ is measurable and $\mu(Z)=1$. Since $y\in \Int \widehat{X}$, there exists $r_y>0$ so small that $B(y,r_y)\sbt\widehat{X}$. Consequently, for all $r\in(0,r_y]$, the \emph{closest} approaches to $y$ using the map $\widehat{T}$ are the same points as for the map $T$ and only the numbers of the iterates may, and usually do, differ. This observation together with \eqref{eq:ind1} gives 
\[
\widehat{\tau}_{B(y,r)}(x)=\tau_{B(y,r)}(x)\frac{l}{A_l(x)}
\]
 for all $r>0$ small enough. Because of \eqref{eq:ind2}, for every $\varepsilon>0$ and every $x\in Z$, there exists $r_y(x,\varepsilon)\in(0,r_y]$ such that 
\[
\widehat{\tau}_{B(y,r)}(x)\left(\frac{1}{\mu(\widehat{X})}- \varepsilon\right)
\le \tau_{B(y,r)}(x) 
\le \widehat{\tau}_{B(y,r)}(x)\left(\frac{1}{\mu(\widehat{X})}+\varepsilon\right)
\]
for all $r\in (0,r_y(x,\varepsilon)]$. Multiplying both sides of this inequality by $\mu(B(y,r))$ and recalling that $\widehat{\mu}(B(y,r))=\mu(B(y,r))/\mu(\widehat{X})$, gives the following
\begin{align*}
&\widehat{\tau}_{B(y,r)}(x)\widehat{\mu}(B(y,r)) \left(1-\varepsilon\mu(\widehat{X})\right) 
=\widehat{\tau}_{B(y,r)}(x)\left(\frac{1}{\mu(\widehat{X})}- \varepsilon\right) \widehat{\mu}(B(y,r))\mu(\widehat{X})\\
&\le \tau_{B(y,r)}(x)\mu(B(y,r)) \\
&\le \widehat{\tau}_{B(y,r)}(x)\left(\frac{1}{\mu(\widehat{X})}+ \varepsilon\right) \widehat{\mu}(B(y,r))\mu(\widehat{X})
= \widehat{\tau}_{B(y,r)}(x)\widehat{\mu}(B(y,r)) \left(1+\varepsilon\mu(\widehat{X})\right).
\end{align*}
Thus, letting $\varepsilon\to 0$ finishes the proof.
\end{proof}

\noindent As an immediate consequence of this lemma, we get the following. 

\begin{cor}\label{c120181123}
Consider an ergodic, metric, measure preserving, dynamical system $(X,T,\mu,d)$; in particular $\mu$ is a Borel probability measure on $X$. Assume that $\mu$ is not supported on any periodic orbit of $T$; equivalently, $\mu$ is atomless. Let $\widehat{X}\subset X$ be a measurable subset of $X$ with positive measure $\mu$. Assume that $y\in \Int \widehat{X}$. Then
\begin{equation}
\limsup_{r\to 0} \tau_{B(y,r)}(x)\cdot \mu(B(y,r))
=\limsup_{r\to 0} \widehat{\tau}_{B(y,r)}(x)\cdot \widehat{\mu}(B(y,r))
=\limsup_{r\to 0} \widehat{\tau}_{B(y,r)}(\widehat{T}(x))\cdot \widehat{\mu}(B(y,r))
\end{equation}
and 
\begin{equation}
\lim\inf_{r\to 0} \tau_{B(y,r)}(x)\cdot \mu(B(y,r))
=\liminf_{r\to 0} \widehat{\tau}_{B(y,r)}(x)\cdot \widehat{\mu}(B(y,r))
=\liminf_{r\to 0} \widehat{\tau}_{B(y,r)}(\widehat{T}(x))\cdot \widehat{\mu}(B(y,r))
\end{equation}
for $\mu$--a.e. $x\in X$, where $\widehat{T}(x)$, the first entry time to $\widehat{X}$, defined as $T^n(x)$, where $n\ge 1$ is a minimal integer such that $T^n(x)\in \widehat{X}$, is well defined for a $\mu$--a.e. $x\in X$ since the measure $\mu$ is ergodic and $\mu(\widehat{X})>0$.

In particular, if 
$$
\limsup_{r\to 0}\widehat{\tau}_{B(y,r)}(x)\cdot\widehat{\mu}(B(y,r))
=+\infty
$$
for $\widehat{\mu}$--a.e. $x\in \widehat{X}$, then
$$
\limsup_{r\to 0} \tau_{B(y,r)}(x)\cdot \mu(B(y,r))=+\infty
$$
for $\mu$--a.e. $x\in X$.
\end{cor}

We now shall take fruits of this corollary and the previous section. Let $(X,d)$ be a metric space and let $T\colon X\to X$ be a continuous map. Fix $z\in X$. Assume that $z$ has a compact neighborhood, call it $W(z)$. We say that the map $T$ is of local IFS type at $z$ if there exists a closed set $\Ga_z\sbt W(z)$ such that
\begin{enumerate}
\item $z\in \Int\(\Ga_z\)$. 

\,\item For every integer $n\ge 1$ there exists $\a_n$, a countable partition of $\tau_{\Ga_z}^{-1}(n)$ ($\tau_{\Ga_z}(x):=\min\{n\in\N\cup\{+\infty\}:T^n(x)\in \Ga_z\}$ is now the first return time of $x\in\Ga_z$ to $\Ga_z$ under $T$), such that
\begin{enumerate}
\,\item For each $A\in \a_n$ the map $T^n|_A\colon A\to\Ga_z$ is a homeomorphism,

\, \item There exists $s\in(0,1)$ such that for every integer $n\ge 1$ and every $A\in\a_n$, the map $\phi_A:=\(T^n|_A\)^{-1}\colon \Ga_z\to A$ is a contraction and its Lipschitz constant does not exceed $s$. 
\end{enumerate}
\end{enumerate}
Obviously the maps 
$$
\cS:=\big\{\phi_A:n\in \N,\, A\in \tau_{\Ga_z}^{-1}(n)\big\}
$$
form an IFS. Therefore, as an immediate consequence of Theorem~\ref{mainthm_2} and Corollary~\ref{c120181123}, we get the following.

\begin{thm}\label{t120180815}
Let $(X,d)$ be a metric space and let $T\colon X\to X$ be a continuous map. Fix $z\in X$. Assume that the map $T$ is of local IFS type at $z$. Assume that $\mu$ is a Borel probability $T$\nobreakdash--invariant measure on $X$ such that
\begin{enumerate}
\item $\mu\(\Ga_z\)>0$ and 

\item $\mu_{\Ga_z}=\hat\mu_f$ for some H\"older continuous summable potential $f$ defined on the symbol space $E^{\N}$ generated by the IFS $\cS$.

\item There exist constants $\alpha>0$ and $C\ge 1$ such that
\begin{equation}\label{420180614C}
\mu(B(z,r))\le Cr^\alpha
\end{equation}
for all radii $r>0$. 
\end{enumerate}
Then for the dynamical system $\(T,\mu\)$ we have that
\begin{equation}\label{eq:rest}
\limsup_{r\to 0} \tau_{B(z,r)}(x)\cdot\mu(B(z,r)) = +\infty, \mbox{\quad for $\hat\mu_f$--a.e. $x\in\mathring{J}_\cS$.}
\end{equation}
\end{thm}

\section{Examples}\label{Examples}
In this section we describe several classes of examples that fulfil the hypotheses of Theorem~\ref{t120180815}. Given a measurable dynamical system $T\colon X\to X$ preserving a probability measure $\mu$, a measurable set $F\subset X$ with
$\mu(F)>0$ and a function $g\colon X\to\R$, we define the function 
$$
g_F\colon X\to\R
$$
by the formula
\begin{equation}\label{6020180828}
g_F(x):=\sum_{k=0}^{\tau_F(x)-1}g\circ T^k(x).
\end{equation}

{\bf Example~A.}\, 
The first class of examples is formed by distance expanding maps $T$ considered in Section~\ref{sec:def} along with invariant measures being equilibrium/Gibbs states of H\"older continuous potentials $\psi$. Indeed, fixing an element $R$ of a Markov partition and taking as the function $f$ of Theorem~\ref{t120180815} the function $\(\psi-\P(\psi)_R\)$ given by \eqref{6020180828}, because of formula \eqref{420180614} the hypotheses of Theorem~\ref{t120180815} are satisfied for all points $z$, and this theorem applies to give us the desired limsup result. So, we have rediscovered Theorem~\ref{mainthm} from a more general, but simultaneously, a more complex perspective. 
\smallskip

{\bf Example~B.}\,  
We now shall describe a large class of dynamical systems  being  multimodal smooth maps of an interval for which Theorem~\ref{t120180815} applies.  
   
We start with the definition of the class of dynamical systems and potentials we consider.

\begin{dfn}\label{dami1}
Let $I = [0,1]$ be the closed interval.  Let $T\colon I \to I$ be a $C^3$ differentiable map with the  following properties:

\, \begin{enumerate}
\item[(a)] $T$ has only finitely  many maximal closed intervals of monotonicity; or equivalently  ${\rm Crit}(T) = \{x \in I \hbox{ : } T'(x) = 0\}$, the set of all critical points of $T$, is finite.

\, \item[(b)] The dynamical system $T\colon I \to I$ is topologically exact, meaning that for every non-empty subset $U$ of $I$ there exists an integer $n \geq 0$ such that $T^n(U) = I$.

\, \item[(c)] All critical points are non-flat.

\, \item[(d)]  $T$  is a \emph{topological Collet--Eckmann map}, meaning that 
$$
\inf\big\{\left(|(T^n)'(x)| \right)^{1/n} \hbox{ : } T^n(x)=x \hbox{ for }  n \geq 1 \big\} > 1
$$
where the infimum is taken over all integers $n \geq 1$ and all fixed points of $T^n$.
\end{enumerate}
We then call $T$ a topologically exact topological Collet--Eckmann map (\emph{abbr.} teTCE).
If (c) and (d) are relaxed and only (a) and (b) are assumed then $T$ is called a topologically exact multimodal map.

We next recall the following definition.
 
\begin{dfn}\label{d2mi4} 
An interval $V \subset I$ is called a \emph{nice set} for a multimodal map $T\colon I \to I$ if
$$
 \hbox{int}(V) \cap \bu_{n=0}^\infty T^n(\partial V) = \emptyset
$$
\end{dfn}
 
The proof of the following theorem is both standard and straightforward; it can be found in many sources, see e.g. \cite{PU-Esc}.

\begin{thm}\label{t3mi4}
If $T\colon  I \to I$ is topologically exact multimodal map then for every point $\xi \in (0,1)$ and every $R > 0$ there exists a nice set $V \subset I$ 
such that $\xi \in V \subset B(\xi, R)$.  
\end{thm}

We set 
$$
{\rm PC}(T) := \bu_{n=1}^\infty T^n({\rm Crit}(T))
$$ 
and call this the \emph{postcritical} set of $T$.  We say that the map $T\colon I \to I$
is \emph{tame} if
$$
\overline {{\rm PC}(T)} \neq I.
$$
\end{dfn}

\noindent
Given a set $F\sbt I$ and an integer $n\ge 0$, we denote by $\cC_F(n)$
the collection of all connected components of $T^{-n}(F)$. 
From their definitions, nice sets enjoy the following property.

\begin{thm}\label{t1mi6}
If $V$ is a nice set for a multimodal map $T\colon I \to I$, then for every integer $n \geq 0$
and every $U \in \cC_V(n)$ either 
$
U \cap V = \emptyset \  \hbox{ or } \  U \subset V.
$
\end{thm}

From now on throughout this section we assume that $T\colon  I \to I$ is a 
tame teTCE map.  Fix a point $\xi \in I \backslash \overline {{\rm PC}(T)}$.
By virtue of Theorem \ref{t3mi4} there is a nice set $V$ such that 
$$
\xi \in V
\  \hbox{ and  }  \
2V \cap
\overline {\PC(T)} = \emptyset.
$$
The nice set $V$  canonically gives rise to a countable alphabet conformal IFS in the sense considered in the previous sections of the present paper. Namely, put 
$$
\cC_V^*:=\bu_{n=1}^\infty\cC_V(n).
$$
For every $U\in \cC_V^*$ let $\tau_V(U)\ge 1$ the unique integer $n\ge 1$
such that $U\in \cC_V(n)$. Put further
$$
\phi_U:=f_U^{-\tau_V(U)}\colon V\to U
$$
and keep in mind that
$$
\phi_U(V)=U.
$$
Denote by $E_V$ the subset of all elements $U$ of $\cC_V^*$ such that 

\begin{enumerate}
\item[(a)] $\phi_U(V)\subset V$, 
\item[(b)] $f^k(U)\cap V=\emptyset$ \, for all \, $k=1,2,\ldots,\tau_V(U)-1$.
\end{enumerate}


\begin{thm}\label{t120180914}
If $V$ is a nice set for a tame teTCE map $T\colon I \to I$ with $\ov V\cap\overline {{\rm PC}(T)}=\es$, then the collection 
$$
\mathcal S_V:=\{\phi_U \colon V\to V\}
$$
of all such inverse branches obviously forms a conformal IFS satisfying the Strong Open Set Condition. In particular, the elements of $\cS_V$ are formed by all  inverse branches of the first return map $f_V\colon V\to V$. 
\end{thm}

The following theorem collects together some fundamental results of 
\cite{HK-MZ}, \cite{HK-ETDS}, and \cite{LR} telling us how nice is the dynamical system generated by the map $T$ and potential $\psi$. 
 
\begin{thm}\label{t1mi3} If $T\colon I \to I$ is a topologically exact multimodal map and $\psi\colon I \to \mathbb R$ is a H\"older continuous potential, then
\begin{enumerate}
\item[(a)] there exists a Borel probability eigenmeasure 
$m_\psi$ for the dual operator $\mathcal L_{\psi}^*$ whose corresponding eigenvalue is equal to $e^{\P(\psi)}$. It then follows that ${\rm supp}(m_\psi) = I$.  
\item[(b)] there exists a unique Borel $T$--invariant probability measure
 $\mu_\psi$ on $I$ absolutely continuous with respect to $m_\psi$. Furthermore, $\mu_\psi$ is equivalent to $m_{\psi}$.
\end{enumerate}
\end{thm}

\noindent Of our most direct interest is item (b) of this theorem which induces a metric dynamical system $(T\colon I\to I, \mu_\psi)$. Because of Theorem~\ref{t120180914} above and Proposition~5.4.8 in \cite{PU-Esc}, our Theorem~\ref{t120180815} applies for this dynamical system, and along with Proposition~\ref{p1sam1}, it yields the following.

\begin{thm}\label{t12018_09_13}  
If $T\colon I \to I$ is a tame teTCE map and $\psi\colon I \to \mathbb R$ is a H\"older continuous potential, then for the dynamical system $\(T\colon I\to I,\mu_\psi\)$, we have that
\begin{equation}\label{eq:restE}
\limsup_{r\to 0} \tau_{B(y,r)}(x)\cdot\mu_\psi(B(y,r)) = +\infty, 
\end{equation}
$\mu_\psi$--a.e. $y\in I$ and $\mu_\psi$--a.e. $x\in I$.
\end{thm}

{\bf Example~C.}\, 
We now shall describe a large class of dynamical systems  being rational functions of the Riemann sphere $\oc$ for which Theorem~\ref{t120180815} applies. 
 Let $f\colon \oc\to\oc$ be a rational function of degree $d\ge 2$. Let $J(f)$ denote the Julia set of $f$. 
Let $\psi\colon \oc\to\R$ be a H\"older continuous potential. We say that $\psi\colon\oc\to\R$ has a \emph{pressure gap} if
\begin{equation}\label{1_2017_03_22}
n\P(\psi)-\sup\(\psi_n\)>0  
\end{equation}
for some integer $n\ge 1$, where $\P(\psi)$ denotes the ordinary topological pressure of $\psi|_{J(f)}$ and the Birkhoff's $n$th sum $\psi_n$ is also considered as restricted to $J(f)$. 

We would like to mention that \eqref{1_2017_03_22} always holds (with all $n\ge 0$ sufficiently large) if the function $f\colon\oc\to\oc$ restricted to its Julia set is expanding, which is then also frequently referred to as hyperbolic. 

The probability invariant measure we are interested in comes from the following.

\begin{thm}[\cite{DU}]\label{t1ma1}
If $f\colon\oc\to\oc$ is a rational function of degree $d\ge 2$ and if $\psi\colon\oc\to\R$ is a H\"older continuous potential with a pressure gap, then $\psi$ admits a unique equilibrium state $\mu_\psi$, i.e. a unique Borel probability $f$-invariant measure on $J(f)$ such that
$$
\P(\psi)={\rm h}_{\mu_\psi}(f)+\int_{J(f)}\psi\,d\mu_\psi.
$$
In addition, 

\,

\begin{itemize}
\item [(a)] the measure $\mu_\psi$ is ergodic, in fact K-mixing, and (see \cite{SUZ_I}) enjoys further finer stochastic properties.

\item [(b)] The Jacobian 
$$
J(f)\ni z\longmapsto \frac{d\mu_\psi\circ f}{d\mu_\psi}(z)\in (0,+\infty)
$$ 
is a H\"older continuous function.
\end{itemize}
\end{thm}

Let
$$
\Crit(f):=\{c\in\oc:f'(c)=0\}
$$
be the set of all critical (branching) points of $f$. As in the case of interval maps set
$$
\PC(f):=\bu_{n=1}^\infty f^n(\Crit(f))
$$
and call it the postcritical set of $f$.

In \cite{PU_tame}, as in the previous class of examples, a rational function $f\colon\oc\to\oc$ was called \emph{tame} if
$$
J(f)\sms \ov{\PC(f)}\ne\es.
$$
Likewise, following \cite{SkU}, we adopt the same definition for (transcendental) meromorphic functions $f\colon\C\to\oc$.

\begin{rem}
Tameness is a very mild hypothesis and there are many classes of maps for which these hold.  These include:
\begin{enumerate}
\item Quadratic maps $z \mapsto z^2 + c$ for which the Julia set is not contained in the real line;
\item Rational maps for which the restriction to the Julia set is expansive which includes the case of expanding rational functions;
\item Misiurewicz maps, where the critical point is not recurrent.  
\end{enumerate}
\end{rem}

The main advantage of dealing with tame functions is that these admit nice sets. Analogously as in the case of interval maps, given a set $F\sbt\oc$ and $n\ge 0$, we denote  by $\cC_F(n)$ 
the collection of all connected components of $f^{-n}(F)$. J. Rivera--Letelier introduced in \cite{Riv07} the concept of nice sets in the realm of the dynamics of rational maps of the Riemann sphere. In \cite{Dob11} N. Dobbs proved their existence for tame meromorphic functions from $\C$ to $\oc$. We quote now his theorem.

\begin{thm}\label{prop:1}
Let $f\colon\C\to\oc$ be a tame meromorphic function. Fix a non-periodic point $z\in J(f)\sms
\ov{\PC(f)}$, $\kappa>1$, and $K>1$. Then for all $L>1$ and 
for all $r>0$ sufficiently small there exists
an open connected set $V=V(z,r)\sbt\C\sms\ov{\PC(f)}$ such that
  \begin{itemize}
  \item[(a)] If $U\in \cC_V(n)$ and $U\cap V\neq \emptyset$, then 
    $U\subseteq V$.
  \item[(b)] If $U\in \cC_V(n)$ and $U\cap v\neq \emptyset$,
    then, for all $w,w'\in U,$ 
    \begin{displaymath}
      |(f^n)'(w)|\ge L
\  \textrm{ and } \
       \frac{|(f^n)'(w)|}{|(f^n)'(w')|}\le K. 
    \end{displaymath}
  \item[(c)] $\overline{B(z,r)}\subset U\subset B(z,\kappa r)\sbt\C\sms\ov{\PC(f)}$.
\end{itemize}
\end{thm}

We now follow the same procedure as in the previous example, see paragraph leading to Theorem~\ref{t120180914}. Define $\cC_V^*$, $\tau_V(U)$ and finally $\phi_U$. Identically as before we arrive at

\begin{thm}\label{t220180914} With hypotheses of Theorem~\ref{prop:1}, the collection 
$$
\cS_V:=\{\phi_U\colon V\to V\}
$$
of all inverse branches forms a conformal IFS. In other words the elements of $\cS_V$ are formed by all holomorphic inverse branches of the first return map $f_V \colon V\to V$. In particular, $\tau_V(U)$ is the first return time of all points in $U=\phi_U(V)$ to $V$.
\end{thm}

As in the case of interval maps, we call a rational function $f\colon\oc\to\oc$ topological Collet--Eckmann (\emph{abbr.} TCE) if
\[
\inf\big\{\left(|(f^n)'(x)| \right)^{1/n}: f^n(x)=x \mbox{ for all }  n \geq 1 \big\} > 1
\]
where the infimum is taken over all integers $n \geq 1$ and all fixed points of $T^n$. There are several other useful characterizations of TCE rational functions, most notably the one commonly referred to as the exponential shrinking property, but we do not really need them in this paper. We can now easily prove the following.

\begin{thm}\label{p1ma3}
If $f\colon\oc\to\oc$ is a tame TCE rational function and $\psi\colon\oc\to\R$ is a H\"older continuous potential then, for the dynamical system $\(f\colon J(f)\to J(f),\mu_\psi\)$, we have that
\begin{equation}\label{eq:restF}
\limsup_{r\to 0} \tau_{B(y,r)}(x)\cdot\mu_\psi(B(y,r)) = +\infty
\end{equation}
for $\mu_\psi$--a.e. $y\in J(f)$ and $\mu_\psi$--a.e. $x\in J(f)$.
\end{thm}

\begin{proof}
Since $f$ is a TCE rational function and $\psi\colon\oc\to\R$ is a H\"older continuous potential, this potential has a pressure gap because of Corollary~1.2
in \cite{IR-R}. So, because of Theorem~\ref{t220180914} and Theorem~\ref{t1ma1} above, and also because of Proposition~5.4.8 in \cite{PU-Esc}, our Theorem~\ref{t120180815} applies for the dynamical system $\(f\colon J(f)\to J(f),\mu_\psi\)$. Along with Proposition~\ref{p1sam1}, this completes the proof.
\end{proof}

{\bf Example~D.}\, 
 Let $f\colon\C\to\oc$ be a meromorphic function. Let $\Sing(f^{-1})$ be the set of all singular points of $f^{-1}$, i. e. the set of all points $w\in\oc$ such that if $W$ is any open connected neighborhood of $w$, then there exists a connected component $U$ of $f^{-1}(W)$ such that the map $f\colon U\to W$ is not bijective. Of course, if $f$ is a rational function, then $\Sing(f^{-1})=f(\Crit(f))$. As in the case of rational functions, we define
\[
\PS(f):=\bu_{n=0}^\infty f^n(\Sing(f^{-1})).
\]
The function $f$ is called \emph{topologically hyperbolic} if
$$
\dist_{\text{Euclid}}(J_f ,\PS(f)) >0,
$$
and it is called \emph{expanding} if there exist $c>0$ and $ \lam>1$
such that
$$
|(f^n)'(z)|\ge c\lam^n 
$$
for all integers $n\ge 1$ and all points $z\in J_f\sms
f^{-n}(\infty)$. Note that every topologically hyperbolic meromorphic
function is tame.  A meromorphic function that is both topologically
hyperbolic and expanding 
is called \emph{hyperbolic}. The meromorphic function $f\colon\C\to\oc$ is
called dynamically {\it semi-regular} if it is of finite order, commonly
denoted by $\rho_f$, and satisfies the following rapid
growth condition for its derivative.  
\begin{equation} \label{eq intro}
|f'(z)|\geq \kappa ^{-1} (1+|z|)^{\al _1} (1+|f(z)|)^{\al_2} \; ,
\quad z\in J_f, 
\end{equation}
with some constant $\kappa >0$ and $\a_1,\a_2$ such that $\al_2 > \max\{-\al _1 ,0\}$.  Set $\a:=\a_1+\a_2$.

\begin{rem}
A particularly simple example of such maps are entire functions $f_\lambda(z) = \lambda e^{z}$
where $\lambda \in (0, 1/e)$ since these maps have an attracting periodic point.  A good reference is \cite{MayUrbETDS}.
\end{rem}

Let $h\colon J_f\to\R$ be a weakly H\"older continuous function 
in the sense of \cite{MayUrb10}. The definition, introduced therein, 
is somewhat technical and we will not provide it here; the simplest example of a weakly H\"older continuous function is the function identically equal to zero. The corresponding function $\psi_{t,0}$ is by no means trivial. Furthermore, each bounded, uniformly locally H\"older function $h\colon J_f\to\R$ is weakly H\"older. Fix $\tau>\a_2$ as required in \cite{MayUrb10}. For $t\in\R$, let
\begin{equation}
  \label{eq:7}
  \psi_{t,h}:=-t\log|f'|_\tau+h
\end{equation}
where $|f'(z)|_\tau$ is the norm, or -- equivalently -- the scaling
factor, of the derivative of $f$ evaluated at a point $z\in J_f$ with
respect to the Riemannian metric
$$
|d\tau(z)|=(1+|z|)^{-\tau}|dz|.
$$
Following \cite{MayUrb10} functions (potentials)
of the form (\ref{eq:7}) are
called \emph{loosely tame}. Let $\cL_{t,h}\colon C_b(J_f)\to C_b(J_f)$ be the
corresponding \emph{Perron--Frobenius operator} given by the formula
$$
\cL_{t,h}g(z):=\sum_{w\in f^{-1}(z)}g(w)e^{\psi_{t,h}(w)}.
$$
It was shown in \cite{MayUrb10} that, for every $z\in J_f$ and for the
function $\1\colon z\mapsto 1$, the limit
$$
\lim_{n\to\infty}\frac1n\log\cL_{t,h}\1(z)
$$
exists and takes on the same common value, which we denote by $\P(t)$
and call \emph{the topological pressure} of the potential
$\psi_t$. The following theorem was proved in \cite{MayUrb10}.

\begin{thm}\label{t1dns111}
If $f\colon\C\to\oc$ is a dynamically semi-regular meromorphic function and
$h\colon J_f\to\R$ is a weakly H\"older continuous potential, then for every
$t>\rho_f/\a$ there exist uniquely determined Borel probability measures $m_{t,h}$ and $\mu_{t,h}$ on $J_f$ with the following properties.


\begin{itemize}
\item[{\rm(a)}] \ $\cL_{t,h}^*m_{t,h}=m_{t,h}$.

\,

\item[{\rm(b)}] \ $\P\(\psi_{t,h}\)=\sup\big\{{\rm h}_\mu(f)+\int\psi_{t,h}\, d\mu:\mu\circ f^{-1}=\mu \  \
  \text{{\rm and }}  \ \int\psi_{t,h}\, d\mu>-\infty\big\}$.

\,

\item[{\rm(c)}] \ $\mu_{t,h}\circ f^{-1}
    =\mu_{t,h}$, $\int\psi_{t,h}\, d\mu_{t,h}>-\infty$, \  and  \
$
{\rm h}_{\mu_{t,h}}(f)+\int\psi_{t,h}\,d\mu_{t,h}=\P\(\psi_{t,h}\).
$

\,

\item[{\rm(d)}] \ The measures $\mu_{t,h}$ and $m_{t,h}$ are equivalent and the Radon--Nikodym derivative $\frac{d\mu_{t,h}}{dm_{t,h}}$ has a nowhere-vanishing H\"older continuous version which is bounded above.
\end{itemize}
\end{thm}  

\noindent Theorem~\ref{prop:1} of course holds and so do the analogs of Theorem~\ref{t220180914} and Proposition~5.5.7 in \cite{PU-Esc}. Thus our Theorem~\ref{t120180815} applies for the dynamical system $\(f\colon J(f)\to J(f),\mu_{t,h})$, and along with Proposition~\ref{p1sam1}, it yields the following.

\begin{thm}\label{p1ma3H}
Let $f\colon\C\to\oc$ be a dynamically semi-regular meromorphic function and let $h\colon J_f\to\R$ be a weakly H\"older continuous potential. If $t>\rho_f/\a$, then for the dynamical system $\(f\colon J(f)\to J(f),\mu_{t,h}\)$, we have that
\begin{equation}\label{eq:restL}
\limsup_{r\to 0} \tau_{B(y,r)}(x)\cdot\mu_{t,h}(B(y,r)) = +\infty, 
\end{equation}
for $\mu_{t,h}$--a.e. $y\in J(f)$ and $\mu_{t,h}$--a.e. $x\in J(f)$.
\end{thm}

\begin{thebibliography}{99}
\bibliographystyle{alpha}

\bibitem{BS}
L.~Barreira and B.~Saussol,
\newblock  Hausdorff dimension of measures via {P}oincar\'{e}
  recurrence,
\newblock {\em Commun. Math. Phys.\/} {\bf 219} (2001), 443--463.

\bibitem{BGI}
C.~Bonanno, S.~Galatolo and S.~Isola,
\newblock  Recurrence and algorithmic information,
\newblock {\em Nonlinearity\/} {\bf 17} (2004), no. 3, 1057–-1074.

\bibitem{Bosh}
M.~D. Boshernitzan,
\newblock  Quantitative recurrence results,
\newblock {\em Invent. Math.\/} {\bf 113} (1993), 617--631.

\bibitem{DU} 
M. Denker and M. Urba\'nski, 
\newblock Ergodic theory of equilibrium states for rational maps, 
\newblock {\em Nonlinearity\/} {\bf 4} (1991), 103--134.

\bibitem{Dob11}
N. Dobbs,
\newblock Nice sets and invariant densities in complex dynamics,
\newblock {\em Math. Proc. Cambridge Philos. Soc\/}., {\bf 150} (2011), 157--165.

\bibitem{Erd}
P.~Erd\"os and P. R\'ev\'esz,
\newblock On the length of the longest head-run,
\newblock {\em Topics in information theory. Coll. Math. Soc. J\'anos Bolyai\/}, Budapest, pp. 219--228 (1975).

\bibitem{GP}
S. Galatolo and P. Peterlongo,
\newblock Long hitting time, slow decay of correlations and arithmetical properties,
\newblock{\em Discrete and Cont. Dynam. Systems.\/} {\bf 27} (2010), no. 1, 185--204.

\bibitem{GRS}
S. Galatolo, J. Rousseau and B. Saussol,
\newblock Skew products, quantitative recurrence, shrinking targets and decay of correlations, 
\newblock{\em Ergodic Theory Dynam. Systems.\/} {\bf 35} (2015), 1814--1845.

\bibitem{HK-MZ}
F. Hofbauer and G. Keller,
\newblock Ergodic properties of invariant measures for piecewise monotonic transformations,
\newblock{\em Mathematische Zeitschrift.\/} {\bf 180} (1982), no. 1, 119--140.

\bibitem{HK-ETDS}
F. Hofbauer and G. Keller,
\newblock Equilibrium states for piecewise monotonic transformations, 
\newblock{\em Ergodic Theory Dynam. Systems.\/} {\bf 2} (1982), no. 1, 23--43.

\bibitem{HLV}
N. Haydn, Y. Lacroix and S. Vaienti
\newblock Hitting and return times in ergodic dynamical systems, 
\newblock{\em Ann. Probab.\/} {\bf 33} (2005), no. 5, 2043--2050.


\bibitem{IR-R}
I. Inoquio-Renteria and and J. Rivera-Letelier,
\newblock A Characterization of hyperbolic potentials of rational maps,
\newblock{\em Bull. Braz. Math. Soc. New Series. \/} {\bf 43} (2012), 99--127. 


\bibitem{Ju}
A.~Junqueira,
\newblock  Quantitative recurrence for generic homeomorphisms,
\newblock {\em Proc. Amer. Math. Soc.\/} {\bf 145} (2017), no. 11, 4751-–4761.

\bibitem{KU-NCPII}
J. Kotus and M. Urba\'nski, 
\newblock Ergodic Theory, Geometric Measure Theory, Conformal Measures and the Dynamics of Elliptic Functions,
\newblock {\em in preparation.\/} 


\bibitem{LR}
H. Li and J. Rivera-Letelier, 
\newblock Equilibrium States of Weakly Hyperbolic One-Dimensional Maps for Hölder Potentials,
\newblock {\em Communications in Mathematical Physics.\/}, {\bf 328} (2014), 397--419.



\bibitem{MU_Israel}
D. Mauldin and M. Urba\'nski, Gibbs states on the symbolic space over an
infinite alphabet, 
\newblock{\em Israel. J. of Math.\/}, {\bf 125} (2001), 93--130.

\bibitem{MU_GDMS}
D. Mauldin and M. Urba\'nski,
\emph{Graph Directed Markov Systems: Geometry and Dynamics of Limit Sets}, Cambridge University Press (2003).

\bibitem{MayUrbETDS}
V. Mayer and M. Urba\'nski, 
\newblock Geometric Thermodynamical Formalism and Real Analyticity for Meromorphic Functions of Finite Order, 
\newblock {\em Ergod. Th. \& Dynam. Sys.\/} {\bf 28} (2008), 915--946.

\bibitem{MayUrb10}
V. Mayer, M. Urba\'nski, 
\newblock Thermodynamical Formalism and Multifractal Analysis for 
Meromorphic Functions of Finite Order, 
\newblock {\em Memoirs of AMS.\/} {\bf 203} (2010), no. 954. 

\bibitem{Pa}
\L{}.~Pawelec,
\newblock  Iterated logarithm speed of return times,
\newblock {\em Bull. Aust. Math. Soc.\/} {\bf 96} (2017) no. 3, 468--478.

\bibitem{PUZI}
\L{}.~Pawelec, M.~Urba\'nski and A.~Zdunik,
\newblock  Thin annuli property and exponential distribution of return times for~weakly~Markov systems,
\newblock {\em preprint.\/}




\bibitem{Rams}
T.~{Persson} and M.~{Rams},
\newblock  On Shrinking Targets for Piecewise Expanding Interval
  Maps,
\newblock{\em Ergodic Theory Dynam. Systems\/} {\bf 37} (2017), 646--663.

\bibitem{PU-Esc}
M. Pollicott and M. Urba\'nski,
\newblock{\em Open Conformal Systems and Perturbations of Transfer Operators},
(Lecture Notes in Mathematics 2206, Springer, 2018). 

\bibitem{PU_tame}
F. Przytycki and M. Urba\'nski,
\newblock  Rigidity of tame rational functions,
\newblock {\em Bull. Pol. Acad. Sci., Math.\/}, {\bf 47.2} (1999), 163--182.

\bibitem{PUbook}
F.~Przytycki and M.~Urba\'{n}ski,
\newblock {\em Conformal Fractals: Ergodic Theory Methods\/}, (Cambridge
University Press, 2010).

\bibitem{Riv07}
J. Rivera-Letelier,
\newblock A connecting lemma for rational maps satisfying a no-growth condition,
\newblock {\em Ergodic Theory Dynam. Systems\/}, {\bf 27} (2007), 595--636.

\bibitem{SkU}
B. Skorulski and M. Urba\'nski,
\newblock Finer Fractal Geometry for Analytic Families of Conformal Dynamical
Systems, 
\newblock {\em Dynamical Systems 29 (2014)\/}, 369--398.
  
\bibitem{SUZ_I}
M. Szostakiewicz, M. Urba\'nski and A. Zdunik,
\newblock Fine Inducing and Equilibrium Measures for Rational Functions of the Riemann Sphere, 
\newblock {\em Israel Journal of Math.\/}, 210 (2015), 399--465.



\end{thebibliography}

\end{document}